\documentclass[11pt,reqno]{amsart}
\usepackage{geometry}                % See geometry.pdf to learn the layout options. There are lots.
\usepackage[parfill]{parskip}    % Activate to begin paragraphs with an empty line rather than an indent
\usepackage{graphicx}
\usepackage{amssymb}
\usepackage{epstopdf}
\usepackage{xcolor}
\usepackage{amsthm}
\usepackage{amsmath}
\usepackage{hyperref}
\usepackage{comment}
\usepackage{tikz}
\usepackage[colorinlistoftodos]{todonotes}
\usepackage{mathrsfs}
\usepackage[normalem]{ulem}

\usepackage{bm}

\newenvironment{proof1}[1]{\begin{trivlist} \item[] {\em Proof of #1:}}{\newline \textcolor{white}{.}\hfill $\Box$
                      \end{trivlist}}

\theoremstyle{plain}

\newtheorem{thm}{Theorem}[section]
\newtheorem{lem}{Lemma}[section]
\newtheorem{prop}{Proposition}[section]

\theoremstyle{definition}

\newtheorem{remark}{Remark}[section]

\newtheorem{defn}{Definition}[section]

\newcommand{\Sw}{\Omega(\w)}

\newcommand{\w}{\bm{w}}
\newcommand{\pa}{\partial}

\newcommand{\eps}{\epsilon}
\newcommand{\R}{\mathbb{R}}
\renewcommand{\epsilon}{\varepsilon}

\newcommand{\Area}{\operatorname{Area}}
\newcommand{\dist}{\operatorname{dist}}

\newcommand{\D}{\mathscr{D}}
\newcommand{\Q}{\mathscr{N}}
\newcommand{\mS}{\mathscr{S}}
\newcommand{\G}{\mathscr{G}}

\definecolor{yaizaColor}{RGB}{0, 153, 153}

\title[Courant Sharp Neumann eigenvalues of Chain Domains]{Uniform upper bounds on Courant sharp Neumann eigenvalues of chain domains}

\author[T. Beck]{Thomas Beck}
\email{tbeck7@fordham.edu}
\address{Department of Mathematics, Fordham University \\ 
407 John Mulcahy Hall, 441 E. Fordham Road, \\
Bronx, NY 10458-5165}

\author[Y. Canzani]{Yaiza Canzani}
\email{canzani@email.unc.edu}
\address{Department of Mathematics, University of North Carolina at Chapel Hill \\ CB\#3250 Phillips Hall \\ Chapel Hill, NC 27599}

\author[J.L. Marzuola]{Jeremy L. Marzuola}
\email{marzuola@math.unc.edu}
\address{Department of Mathematics, University of North Carolina at Chapel Hill \\ CB\#3250 Phillips Hall \\ Chapel Hill, NC 27599}

\begin{document}

\maketitle

\begin{abstract}
We obtain upper bounds on the number of nodal domains of Laplace eigenfunctions on chain domains with Neumann boundary conditions. The chain domains consist of a family of planar domains, with piecewise smooth boundary, that are joined by thin necks. Our work does not assume a lower bound on the width of the necks in the chain domain. As a consequence, we
prove an upper bound on the number of Courant sharp eigenfunctions that is independent of the widths of the necks.
\end{abstract}

%%%%%%%%%%%%%%%%%%%%%%%%%%%%%%%%%%%%%%%%%
%%%%%%%%%%%%%%%%%%%%%%%%%%%%%%%%%%%%%%%%%
\section{Introduction}
%%%%%%%%%%%%%%%%%%%%%%%%%%%%%%%%%%%%%%%%%
%%%%%%%%%%%%%%%%%%%%%%%%%%%%%%%%%%%%%%%%%

A long-studied question is to what extent the spectrum of the Laplacian  interacts with the geometry of the domain on which it is defined \cite{kac1966can}. In this article, we study {\it chain domains} which consist of a family of disjoint bounded planar domains, with piecewise smooth boundaries, that are connected by thin necks (see Figure \ref{fig:domain}).  For these chain domains, we prove bounds on the number of nodal domains of the associated Laplace eigenfunctions when Neumann boundary conditions are imposed. 

Let $\Omega \subset \R^2$ be a bounded domain with piecewise smooth boundary and write $0=\mu_1 < \mu_2 \leq \ldots$ for the Laplace eigenvalues of the Neumann problem
\begin{equation}\label{e:ev}
\begin{cases}
\Delta u_m = -\mu_m u, \quad &\text{in}\; \;\Omega,\\
\partial_n u_m=0, \quad & \text{on}\; \; \partial \Omega.
\end{cases}
\end{equation}
The Courant Nodal Domain Theorem  \cite{courant} asserts that for all $ m=1,2, \ldots$,
\begin{equation}\label{e:courant}
\nu(u_m)\leq m, 
\end{equation}
where $\nu(u_m)$ is the number of nodal domains of $u_m$. That is, $\nu(u_m)$ is the number of connected components of $\{x\in \Omega: u_m(x)\neq 0\}$.
If $\nu(u_m)=m$, then $u_m$ is said to be a {\it Courant sharp eigenfunction} and $\mu_m$ a {\it Courant sharp eigenvalue}.

In this article, we study the case in which $\Omega$ is a chain domain and prove an asymptotic upper bound on $\nu(u_m)$ that, in particular, provides an upper bound on the values of $\mu_m$ that can be Courant sharp. The bound we prove is independent of a lower bound of the widths of the necks in the chain domain, hence notably does not depend upon the cut-distance of $\partial \Omega$.

A chain domain  consists of a family of disjoint bounded planar domains, $\{\D_\ell\}_{\ell=1}^{M}$, and a family of thin necks joining the domains. The planar domains  have smooth boundaries except for a finite number of vertices. 

Since we are interested in understanding how our estimates respond to the width of the necks shrinking to zero, we actually think of a chain domain as a structure built around a \emph{skeleton} comprised of the domains $\{\D_\ell\}_{\ell=1}^{M}$ joined by a family of curves  $\{\Gamma_i\}_{i=1}^N$. Then, the skeleton is ``thickened" into what we call a \emph{base domain}, $\Omega$, by turning each curve $\Gamma_i$ into a neck that joins two domains. Formally, associated to each curve $\Gamma_i$ there is a smooth homotopy $\G_i:[0,L_i]\times [-1,1] \to \R^2$ such that $\Gamma_i=\G_i([0,L_i]\times \{0\})$ 
and each of the arcs $\G_i(\{0\}\times [-1,1])$ and $\G_i(\{L_i\}\times [-1,1])$ belong to the boundaries of the two domains being joined by $\Gamma_i$. The curve $\Gamma_i$ is then thickened to the neck  $\Q_i:=\G_i([0,L_i]\times [-1,1])$. The base domain is the set $\Omega=\bigcup_{\ell=1}^{M}  \D_\ell\cup \bigcup_{i=1}^N\Q_i.$ See Definition \ref{defn:base} for a detailed description.

Given a base domain $\Omega$, we allow for the width of each neck to vary by working with a family of neck widths $\w:=\{I_i\}_{i=1}^N$, where the interval $I_i \subset [-1,1]$ for each $i$. 
Indeed, we define each neck as $\Q_i(\w):=\G_i([0,L_i]\times I_i)$. The \emph{chain domain} is then
$$\Omega(\w)=\bigcup_{\ell=1}^{M}  \D_\ell\cup \bigcup_{i=1}^N\Q_i(\w).$$

Note that this definition of a chain domain includes polygonal figures with smooth edges. A toy model with a single neck is depicted in Figure~\ref{fig:toy}. See Definition \ref{defn:domain} for a precise definition of a chain domain and Figure \ref{fig:domain} for an illustration of such.

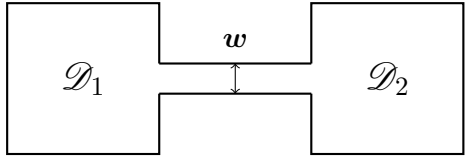
\begin{figure}[h]
    \centering
\begin{tikzpicture}
\draw[thick] (-1,0.2) -- (1,0.2);
\draw[thick] (-1,-0.2) -- (1,-0.2);
\draw[thick] (1,0.2) -- (1,1) -- (3,1) -- (3,-1) -- (1,-1) -- (1,-0.2);
\draw[thick] (-1,0.2) -- (-1,1) -- (-3,1) -- (-3,-1) -- (-1,-1) -- (-1,-0.2);
\node[font=\Large] at (-2,0) {$\D_1$};
\node[font=\Large] at (2,0) {$\D_2$};
\draw[<->] (0,-0.2)--(0,0.2);
\node at (0,0.5) {$\w$};
\end{tikzpicture}
    \caption{A simple chain domain, $\Omega(\w)=\D_1 \cup \Q  \cup \D_2$, in which two {squares, $\D_1, \D_2$,} {of side length $2$} are joined by the neck {$\Q:=[-1,1]\times \w$, with $\w= (-w/2,w/2)$}.}
    \label{fig:toy}
\end{figure}

Our main result is the following.
\begin{thm} \label{thm:Main}
Let $\Omega$ be a base domain. There exists $C>0$ such that for every family of neck widths, $\w$, the following holds.
If $\{u_m(\w)\}_m$ and $\{\mu_m(\w)\}_m$ are the Neumann eigenfunctions and eigenvalues for the chain domain $\Omega(\w)$, then
\begin{align}\label{e:main ub}
      \underset{m\to\infty}{\lim \sup}\,\frac{\nu(u_m(\w))}{m} \leq  \frac{4}{\lambda_1(\mathbb{D})}.
\end{align}
Moreover,  for every Courant sharp eigenvalue $\mu_m(\w)$,
\begin{equation}\label{e:main ub 2}
\Area(\Omega(\w))\mu_m(\w)< C.
\end{equation}
\end{thm}

Here, $\lambda_1(\mathbb{D})$ is the first Dirichlet eigenvalue of the unit disc, $\mathbb{D}\subset \R^2$. It is known that $\lambda_1(\mathbb{D}) =j_{0,1}^2$, where $j_{0,1}$ is the first positive zero of the Bessel function $J_0$. Since ${4}/\lambda_1(\mathbb{D}) \approx 0.692 < 1$,  the upper bound in \eqref{e:main ub} implies that there are only finitely many Courant sharp eigenfunctions. The estimate in \eqref{e:main ub 2} then gives an upper bound on the eigenvalue of such an eigenfunction, independent of the choice of neck-widths $\w$.

\begin{remark}
In Section \ref{sec:unif} we introduce five geometric constants $\rho^*, \kappa^*, \delta^*, \tau^*, w^*$ associated to the geometry of the base domain $\Omega$. 
The constant $C$ in \eqref{e:main ub 2} depends only on  $\rho^*, \kappa^*, \delta^*, \tau^*, w^*$. The same constant $C$ can be used for all chain domains whose associated base domains, $\tilde \Omega$, have geometric constants controlled by $\rho^*, \kappa^*, \delta^*, \tau^*, w^*$  as explained in Remark \ref{rem:uniform}. This implies that the number and structure of the planar domains and necks comprising $\Omega$ only play a role in controlling the Courant sharp eigenvalues through how they impact these five geometric constants.
\end{remark}

There are several articles that study  eigenfunction nodal domain count under various boundary conditions (see Section \ref{sec:prior} for an account of these). Directly related to our main result is the work of Gittins and L\'ena \cite{gittins-lena2020}, who also give a bound on $\text{Area}(\Omega)\mu_m$ when $\mu_m$ is Courant sharp. However,  applying their result would require imposing a lower bound on the widths of the connecting necks in the chain. A strength of our work is that we prove bounds on the nodal domain counting function that are independent of the width of the necks, and hence obtain that the oscillatory behavior of the Courant sharp eigenfunctions is driven by the geometry of the planar domains that make up the chain, as stated in \eqref{e:main ub 2}. 
 The bound in \eqref{e:main ub} was proved by Polterovich \cite{polterovich2009} when $\Omega \subset \R^2$  has piecewise real analytic boundary and by L\'ena \cite{lena-pleijel} when the boundary is $C^{1,1}$. Theorem \ref{thm:Main} extends their work to more general domains, including polygonal regions with smooth edges.

When working with a chain domain on which \emph{Dirichlet boundary conditions} are imposed, the spectrum of the Laplacian approaches the spectrum of the disjoint union of the individual planar domains as the widths of the necks connecting the domains decrease to $0$. The corresponding eigenfunctions do not `see' the thin necks and become localized to the connecting planar domains, for example,  see \cite{daners2003dirichlet}. In contrast, when Neumann boundary conditions are imposed, the behavior of the low-energy eigenfunctions is influenced by the presence of the thin necks. Indeed, in this case, the eigenfunctions for the chain behave like the restricted eigenfunctions on the individual smooth domains, but are `connected' along each neck by a function that, to leading order, is an eigenfunction of a one dimensional Schr\"odinger operator \cite{arrieta95,J89,Jim,MJIm}.

In this work, instead of studying low-energy eigenfunctions, we ask whether the intermediate to high energy eigenfunctions have oscillatory behavior that is more closely related to the geometry of the family of planar domains in the chain than to the structure of the thin necks connecting them. Theorem \ref{thm:Main} shows that the behavior of Courant sharp eigenfunctions is driven by that of the planar domains and ignores the structure of the connecting necks.  Interestingly, the oscillatory behavior of eigenfunctions has been known to be unstable to domain perturbations for some time, see \cite{U73,U76}. Indeed, in previous work, the authors explored how boundary perturbations can lower the number of nodal domains of low energy eigenfunctions \cite{beck2022nodal}.  It is thus natural to think that, in general, it should be more challenging to have many internal oscillations in chain domains than in the piecewise smooth domains themselves, but we do not address that question  here.

The problem we study has a natural geometric interpretation of a community detection problem in data analysis.  Many data sets have the structure of stochastic block models, which consist of clusters of sub-networks that are strongly connected with weak connections between them.  It is a natural question in learning and community detection algorithms to determine the number of communities that one should actually look for in a data set.  This study is a first step in understanding how the spectrum of the Laplacian with Neumann boundary conditions may be used to explore such a question in a continuum setting, which can in many ways be seen as a model for large data sets with manifold structure, see \cite{singer2017spectral,trillos2018variational}.  Indeed, we conjecture that a chain domain of $k$ planar domains connected in a line will have its first $k$ eigenfunctions saturating Courant's nodal domain bound by oscillating as much as possible, while the eigenfunctions with sufficiently large eigenvalues will not  saturate the bound.  This article works towards this conjecture, by giving bounds on the nodal domain count that are independent of the thin connections between domains.

%%%%%%%%%%%%%%%%%%%%%%%%%%%%%%%%%%%%%%%%%
\subsection{Prior results}\label{sec:prior}
%%%%%%%%%%%%%%%%%%%%%%%%%%%%%%%%%%%%%%%%%
While we will focus on planar domains with Neumann boundary conditions, the Courant Nodal Domain Theorem (see equation \eqref{e:courant}) also holds for  eigenfunctions with Dirichlet boundary conditions. In \cite{pleijel1956}, Pleijel proved that for planar domains with Dirichlet boundary conditions, there are only a finite number of Courant sharp eigenfunctions. This was extended by Peetre in \cite{peetre1957} to domains on Riemannian surfaces, and by B\'erard and Meyer \cite{berard-meyer1982} to $n$-dimensional Riemannian manifolds that are compact or have smooth boundary. In each case, the result is achieved by obtaining a limiting upper bound on the ratio between the number of nodal domains of the  $m$-th  Dirichlet  eigenfunction, $v_m$, and its index $m$. In two dimensions, this upper bound is
\begin{align} \label{eqn:pleijel}
    \underset{m\to\infty}{\lim \sup}\;\frac{\nu(v_m)}{m} \leq  \frac{4}{\lambda_1(\mathbb{D})},
\end{align}
the same as in Theorem \ref{thm:Main}. Since  $\lambda_1(\mathbb{D}) =j_{0,1}^2$ and ${4}/{j_{0,1}^2} \approx 0.692 < 1$, the upper bound in \eqref{eqn:pleijel} immediately implies that there are only finitely many Courant sharp Dirichlet eigenfunctions.

A key step in the proof of \eqref{eqn:pleijel} is to obtain a lower bound on the area of a nodal domain using the Faber-Krahn Theorem. This theorem can be used in the Dirichlet case because the restriction of a Dirichlet eigenfunction to a nodal domain is the first Dirichlet eigenfunction of that domain. When Neumann boundary conditions are imposed, this is no longer true for nodal domains that touch the boundary. However, by using a different method to count nodal domains near the boundary, Polterovich \cite{polterovich2009} in the two-dimensional case with piecewise real analytic boundaries, and L\'ena \cite{lena-pleijel} in $n$-dimensions with $C^{1,1}$-boundaries, showed that \eqref{eqn:pleijel} continues to hold for Neumann eigenfunctions, and Robin eigenfunctions with non-negative Robin parameter.  This was extended to any sign of the Robin parameter, for domains with $C^{1,1}$-boundaries, in the recent work \cite{hassannezhad2023pleijel}.

While the above results guarantee that there are only finitely many Courant sharp eigenfunctions, they do not provide an upper bound on the corresponding eigenvalue in terms of geometric quantities of the underlying domain. For a planar bounded domain $\Omega$,
a bound on  $\text{Area}(\Omega)\lambda$, when $\lambda$ is a Courant sharp Dirichlet eigenvalue, was established by B\'erard and Helffer \cite{berard-hellfer2016}.  Such a bound is given in terms of dilation invariant quantities involving $\text{Area}(\Omega)$ and $\sup_{\delta>0}\,\frac{1}{\delta}M_{\Omega}(\delta)$, where
\begin{equation}\label{e:M}
     M_{\Omega}(\delta):=\text{Area}\left(\{x\in\Omega:\,\dist(x,\pa\Omega)<\delta \} \right).
\end{equation}
For an open set $\Omega \subset \R^n$ of finite Lebesgue measure, van den Berg and Gittins \cite{vdberg-gittins2016} obtain an upper bound on a Courant sharp Dirichlet eigenvalue in terms of $\text{Area}(\Omega)$ and $ M_{\Omega}(\delta)$.
For a planar domain with $C^2$-boundary, Gittins and L\'ena  prove an upper bound on $\text{Area}(\Omega)\mu$, when $\mu$ is a Courant sharp Neumann eigenvalue \cite{gittins-lena2020}. This bound is given in terms of an upper bound on $ \text{Area}(\Omega)$, the isoperimetric ratio, the curvature of $\partial \Omega$, and  a lower bound on the cut-distance
\begin{align*}
    \inf_{t\in[0,L]}\sup_{\delta>0}\left\{\delta: \;\dist\Big(\gamma(t)+sn(t),\pa\Omega\Big)=s \; \text{for all} \;s\in[0,\delta] \Big.\right\}.
\end{align*}
Here, $\gamma:[0,L] \to \R^2$ is a parameterization of  $\pa\Omega$ and  $n(t)$ is the unit inward normal at $\gamma(t)$. 

Note that as the width $w$ shrinks to $0$, the chain domain $\Sw$ from  Figure \ref{fig:toy} does not fit into the results of Gittins and L\'ena in \cite{gittins-lena2020}. This is because the cut-distance of the neck, $[-1,1]\times (-w/2,w/2)$, 
shrinks to zero when $w$ approaches $0$. Moreover, the sharp corners in $\Sw$ mean that the curvature of the boundary is not bounded from above. Therefore, a major part of our study of the Courant sharp Neumann eigenfunctions of such a chain domain will be focused on the behavior of Neumann eigenfunctions in thin necks and near the vertices of the domain.

\subsection*{Acknowledgements}  T.B. was supported through NSF Grant DMS-2042654. Y.C. was supported by the Alfred P. Sloan Foundation, NSF CAREER Grant DMS-2045494, and NSF Grant DMS-1900519.  JLM acknowledges support from the NSF through NSF grant DMS-1909035 and NSF FRG grant DMS-2152289. 

%%%%%%%%%%%%%%%%%%%%%%%%%%%%%%%%%%%%%%%%%
%%%%%%%%%%%%%%%%%%%%%%%%%%%%%%%%%%%%%%%%%
\section{Chain domains}
%%%%%%%%%%%%%%%%%%%%%%%%%%%%%%%%%%%%%%%%%
%%%%%%%%%%%%%%%%%%%%%%%%%%%%%%%%%%%%%%%%%

As described in the introduction, we study the Neumann eigenvalues and eigenfunctions of a class of planar domains consisting of a number of piecewise smooth domains joined by thin necks. We now give the precise definition of the chain domains under consideration. See Figure \ref{fig:domain} for an example of such a domain.

We define the family of chain domains in three steps. First, we introduce the \textit{skeleton} consisting of a finite number of piecewise smooth connected regions, joined by smooth curves. We then define a \textit{base domain}, by using homotopies to replace the smooth curves by fixed necks. Finally, we restrict the domains of these homotopies to allow for a full family of necks, of arbitrarily small width.
\begin{defn}[Skeleton] \label{defn:skeleton}
Let  $M\in \mathbb N$, $\{K_{i,j}\}_{i,j=1}^M \subset \mathbb N$. Then, the skeleton domain $\mS$ is given by the disjoint union
\begin{align*}
\mS = \left(\bigcup_{\ell=1}^{M} \D_\ell\right) \cup \left(\bigcup_{1\leq i\leq j\leq M}\bigcup_{k=1}^{K_{ij}}\Gamma_{ij,k}\right).
\end{align*}
\begin{itemize}
\item For $1 \leq \ell \leq M$, each $\D_\ell$ is a bounded planar domain with smooth boundary except for a finite number of vertices.

\item For $1\leq i \leq j \leq M$ and $1\leq k \leq K_{ij}$, $\Gamma_{ij,k}$ is a smooth \textit{neck curve} joining $\D_i$ and $\D_j$, parameterized by arc-length, 
$$
\Gamma_{ij,k}: [0,L_{ij,k}]\to \R^2,
$$
with $\Gamma_{ij,k}(0) \in \pa \D_i$ and  $\Gamma_{ij,k}(L_{ij,k}) \in \pa \D_j$.
\end{itemize}
\end{defn}

Next, we endow a skeleton $\mS$ with a family of homotopies that encode how the curves $\{\Gamma_{ij,k}\}$ are thickened into necks.

\begin{defn}[Base domain] \label{defn:base}
    Let  $\mS$ be a skeleton domain as in Definition \ref{defn:skeleton}.  The base domain $\Omega$ associated to  $\mS$ is the disjoint union
    \begin{align*}
    \Omega = \left(\bigcup_{\ell=1}^{M} \D_\ell\right) \cup \left(\bigcup_{1\leq i\leq j\leq M}\bigcup_{k=1}^{K_{ij}}\G_{ij,k}([0,L_{ij,k}]\times(-1,1))\right),
    \end{align*}
where for each $i,j,k$,
    \begin{align*}
        \G_{ij,k}:[0,L_{ij,k}]\times[-1,1] \to \R^2
    \end{align*}
    is a smooth homotopy with the following properties:
    \begin{itemize}
        \item $\G_{ij,k}(s,0) = \Gamma_{ij,k}(s)$ for all $s\in[0,L_{ij,k}]$;
        \item $\G_{ij,k}(0,t)\in\pa \D_{i}$ and $\G_{ij,k}(L_{ij,k},t)\in\pa \D_j$, for all $t\in[-1,1]$, and not coinciding with any vertex of $\D_i$ or $\D_j$;
        \item ${\rm det} [\pa_t\G_{ij,k}(s,t), \pa_s \G_{ij,k}(s,t)]\neq0$ for all $(s,t)$, i.e. the Jacobian for this coordinate system is non-degenerate;
        \item  the images of the $\G_{ij,k}$ are disjoint, and disjoint from the interiors of all the domains $\D_\ell$.
    \end{itemize}
\end{defn}

    Note that since ${\rm det} [\pa_t\G_{ij,k}(s,t), \pa_s \G_{ij,k}(s,t)]\neq0$ for all $(s,t)$, each curve $\G_{ij,k}(\cdot,t)$ meets $\pa\D_i$ and $\pa \D_j$ transversally. We are now ready to introduce the chain domains, in which we allow the widths of the necks to become arbitrarily small.

\begin{defn}[Chain domain] \label{defn:domain}
Let $\Omega$ be a base domain as in Definition \ref{defn:base}.
    We define a collection of widths $\w$ by \begin{align*}
        \w=\{I_{ij,k}:\; 1\leq i,j, \leq M, \; 1\leq k \leq K_{ij}\},
    \end{align*}
    where  each $I_{ij,k}$ is an open interval in $(-1,1)$ containing $0$. The chain domain $\Omega(\w)$ is then given by the disjoint union
    \begin{align*}
    \Omega(\w) = \left(\bigcup_{\ell=1}^{M} \D_\ell\right) \cup \left(\bigcup_{1\leq i\leq j\leq M}\bigcup_{k=1}^{K_{ij}} \Q_{ij,k}(\w)\right).
    \end{align*}
    Here $ \Q_{ij,k}(\w)$ is a \textit{neck} joining $\D_i$ and $\D_j$ given by
    \begin{align*}
         \Q_{ij,k}(\w) = \G_{ij,k}([0,L_{ij,k}]\times I_{ij,k}).
    \end{align*}
    A measure of the minimum width of the neck $ \Q_{ij,k}(\w)$ is defined by
    \begin{align}\label{e:ws}
        w_{ij,k} := \inf_{s\in [0,L_{ij,k}]}|\G_{ij,k}(s,t_2) - \G_{ij,k}(s,t_1)| 
    \end{align}
    for  $I_{ij,k}=(t_1,t_2)$. 
\end{defn}
The above definition includes the case $M=1$, $K_{11}=0$, where $\Omega(\w) = \D_1$ is a bounded domain, with piecewise smooth boundary except for a finite number of vertices, and with no $\w$ dependence. We define the interval $I_{ij,k}$ to contain $0$  to ensure that the chain domain $\Omega(\w)$ always contains the skeleton $\mS$.

\begin{figure}[h]
    \centering
\begin{tikzpicture}
\draw[thick] (-1,0.7) .. controls (-0.2,0.8) and (0.2,0.9) .. (1,0.7);
\draw[thick] (-1,0.4) .. controls (-0.4,0.6) and (0.2,0.3) .. (1,0.4);
\draw[thick] (-1,-0.6) .. controls (-0.7,-0.4) and (0.2,-0.65) .. (1,-0.6);
\draw[thick] (-1,-0.8) -- (1,-0.8);
\draw[thick] (-1,0.7) -- (-1,1);
\draw[thick] (-1,-0.8) -- (-1,-1);
\draw[thick] (-1,0.4) -- (-1,-0.6);
\draw[thick] (1,0.7) -- (1,1.2);
\draw[thick] (1,-0.8) -- (1,-1.1);
\draw[thick] (1,0.4) -- (1,-0.6);
\draw[thick] (1,1.2) -- (3,1.6);
\draw[thick] (3,1.6) .. controls (2.5,1) and (3.4,0) .. (2.8,-1.5);
\draw[thick] (1,-1.1) .. controls (1.6,-0.8) and (2.4,-1.6) .. (2.8,-1.5);
\draw[thick] (-1,1) .. controls (-2,0.7) and (-2.5,1.2) .. (-4,1.3);
\draw[thick] (-4,1.3) .. controls (-4,0) and (-3.6,-1) .. (-3.5,-2);
\draw[thick] (-2.7,-0.4) circle (0.5cm);
\draw[thick] (-3.5,-2) .. controls (-3,-2) and (-2.8,-1.8,) .. (-2.5,-1.5);
\draw[thick] (-2.3,-1.3) .. controls (-2,-1) and (-1.8,-1) .. (-1,-1);
\draw[thick] (-2.5,-1.5) .. controls (-2,-4.4) and (-0.8,-3.6) .. (-0.2,-3.25);
\draw[thick] (-2.3,-1.3) .. controls (-2,-3.8) and (-0.8,-3.5) .. (-0.45,-3.13);
\draw[thick] (-1.2,-2) .. controls (-1.7,-2.5) and (-1.3,-2.9) .. (-1.3,-3);
\draw[thick] (-1.1,-1.75) .. controls (-0.93,-1.7) and (-0.87,-1.65) .. (-0.9,-1.3);
\draw[thick] (-1.1,-1.75) .. controls (-1.15,-1.85) and (-1.17,-1.95) .. (-1.2,-2);
\draw[thick] (-1.3,-3) .. controls (-0.9,-3.5) and (-0.6,-2.5) .. (-0.45,-3.13);
\draw[thick] (-0.2,-3.25) .. controls (0.2,-3.5) and (0.4,-2.5) .. (1,-2.8);
\draw[thick] (-0.9,-1.3) .. controls (-0.6,-1.5) and (-0.3,-1.2) .. (0,-1.3);
\draw[thick] (0,-1.3) .. controls (0.3,-1.7) and (0.6,-1.5) .. (1,-1.7);
\draw[thick] (1,-1.7) .. controls (0.7,-2.2) and (0.8,-2.4) .. (1,-2.8);
\node[font=\Large] at (-1.75,0) {$\D_1$};
\node[font=\Large] at (2,0) {$\D_2$};
\node[font=\Large] at (0,-2.2) {$\D_3$};
\node at (0,1.75) {$\Q_{12,1}(\w)$};
\node at (2.6,-2.1) {$\Q_{12,2}(\w)$};
\node at (-3.2,-3.5) {$\Q_{13,1}(\w)$};
\draw (-0.3,1.45) -- (0,0.7);
\draw (1.8,-1.8) -- (0,-0.7);
\draw (-3.4,-3.2) -- (-2.25,-2.15);
\end{tikzpicture}
    \caption{An example of a domain $\Omega(\w)$ from Definition \ref{defn:domain}}
    \label{fig:domain}
\end{figure}
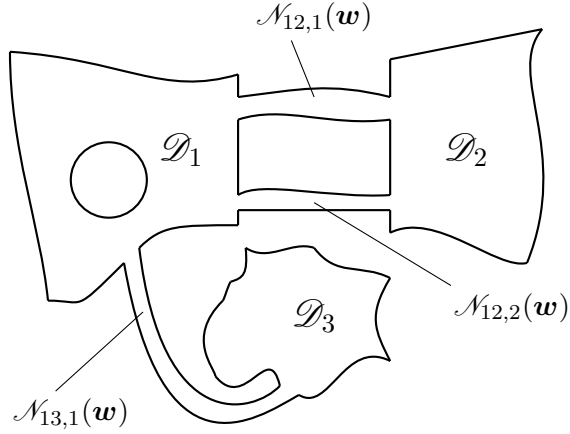

%The above definition includes the case $M=1$, $K_{11}=0$, where $\Omega(\w) = \D_1$ is a bounded domain, with piecewise smooth boundary except for a finite number of vertices, and with no $\w$ dependence. In general, this definition ensures that the upper and lower bounds of each neck are a graph with respect to the same orientation.

%%%%%%%%%%%%%%%%%%%%%%%%%%%%%%%%%%%%%%%%%
 \subsection{Geometric constants}\label{sec:unif}
 %%%%%%%%%%%%%%%%%%%%%%%%%%%%%%%%%%%%%%%%%

Let $\Omega$ be a base domain as in Definition \ref{defn:base}.
We define $$A^* := \sum_{\ell}\text{Area}(\D_\ell)$$ and 
\begin{align*}
    L^* := \sum_{\ell}\text{Length}(\pa\D_\ell) + 2\sum_{i,j,k} \max_{t\in[-1,1]} \text{Length}(\G_{ij,k}([0,L_{ij,k}]\times\{t\})).
\end{align*}
We will bound the Courant sharp Neumann eigenvalues of $\Omega$ in terms of the five geometric constants, $\rho^*, \kappa^*, \delta^*, \tau^*, w^*$ defined as follows.

\begin{itemize}
    \item \emph{Isoperimetric ratio constant, $\rho^*$}: We define $\rho^*>0$ to be the isoperimetric ratio,
$$
\rho^* = \frac{ (L^*)^2}{A^*}.
$$
    \item \emph{Normalized curvature constant, $\kappa^*$}: We define $\kappa^*>0$ so that the curvature of each smooth segment of $\pa \D_\ell$ and the slices $\G_{ij,k}([0,L_{ij,k}]\times \{t\})$ are bounded above by $\kappa^*/L^*$ for all $t\in [-1,1]$ and for all choices of $\ell, i, j, k$.
    \item \emph{Vertex control constant, $\delta^*$}:  The constant $\delta^*>0$ is defined so that the following holds. For all $\ell$ and each vertex  $p \in \pa \D_\ell$, we have that the connected component of  $B(p, L^*\delta^*)\cap \pa\Omega$ containing $p$ consists of two smooth curves joined at $p$, and after possibly rotating, we may assume that their tangent lines at $p$ agree with the  lines $\theta = \tfrac{\pi}{2} \pm\tfrac{\theta_0}{2}$ for some $0 < \theta_0 < \pi$. These are graphs with respect to the $x_1$-axis, contained within the lines $\theta = \tfrac{\pi}{2} \pm\tfrac{\theta_0}{4}$ and $\theta = \tfrac{\pi}{2} \pm\tfrac{3\theta_0}{4}$, and with slope bounded by $1/\delta^*$. Moreover, the same properties hold at the vertices where the sides of $\pa \D_{i}$, $\pa\D_{j}$ meet the curves $\G_{ij,k}(\cdot,t)$, for each $t\in[-1,1]$ and all choices of $i,j,k$.
    \item  \emph{Normalized cut-distance constant, $\tau^*$}: The constant $\tau^*$ is defined so that, for $\eta\leq L^*\delta^*$ and all $\ell$, the cut-distance, given by
    \begin{align*}
    \inf_{u\in K_{\eta,\ell}}\sup_{\delta>0}\left\{\delta: \;\dist\Big(\gamma(u)+sn(u),\pa\D_{\ell}\Big)=s \; \text{for all} \;s\in[0,\delta] \Big.\right\},
\end{align*}
is bounded from below by $\tau^*\eta$. Here, $\gamma:K_{\eta,\ell}\to\R^2$ is a parameterization of $\pa \D_{\ell}$ with the parts of $\pa\D_{\ell}$ in the discs of radius $\eta$ centered at each vertex of $\pa\D_{\ell}$ excluded, for a union of intervals $K_{\eta,\ell}$, and $n(u)$ is the inward unit normal at $\gamma(u)\in\pa\D_{\ell}$.
    
    \item \emph{Neck-width constant, $w^*$}: The constant $w^*$ is defined to provide control on the widths and regularity of the necks  so that for all choices of $i,j,k,$
    \begin{align}\label{e:max-min width}
          & \min_s\frac{\min_{t\in[-1,1]}|\pa_t\G_{ij,k}(s,t)|}{\max_{t\in[-1,1]}|\pa_t\G_{ij,k}(s,t)|}\geq w^*, \qquad  w^*\leq |\pa_s\G_{ij,k}(s,t)| \leq 1/w^*, \\ & \hspace{2cm} \min_{(s,t)}\frac{\left|{\rm det} [\pa_t\G_{ij,k}(s,t), \pa_s \G_{ij,k}(s,t)]\right|}{|\pa_t\G_{ij,k}(s,t)||\pa_s\G_{ij,k}(s,t)|}\geq w^*. \notag
    \end{align}
\end{itemize}

Our main result, Theorem \ref{thm:Main}, produces a bound  for a Courant sharp eigenvalue $\mu_m(\w)$ of the form  $|\Omega(\w)|\mu_m(\w)<C$, where $C$ only depends on $\rho^*, \kappa^*, \delta^*, \tau^*, w^*$. Since these control quantities are invariant under dilations of the base domain $\Omega$, the constant $C$ in the theorem can also be chosen uniformly over dilations of the domain. More generally, we have the following remark.

\begin{remark}[Uniform control over base domains $\Omega$] \label{rem:uniform}
    The constant $C$ in \eqref{e:main ub 2} in Theorem \ref{thm:Main} can be taken to be the same for every chain domain $\tilde \Omega(\w)$ whose associated base domain $\tilde{\Omega}$ has geometric constants $\tilde{\rho}^*, \tilde{\kappa}^*, \tilde{\delta}^*, \tilde{\tau}^*, \tilde{w}^*$ satisfying
    \begin{align*}
\tilde{\rho}^* \leq \rho^*, \quad \tilde{\kappa}^* \leq \kappa^*, \quad \tilde{\delta}^* \geq \delta^*, \quad \tilde{\tau}^* \geq \tau^*, \quad \tilde{w}^* \geq w^*,
    \end{align*}
   where $\rho^*, \kappa^*, \delta^*, \tau^*, w^*$ are the geometric constants associated to the base domain $\Omega$.
\end{remark}

\subsubsection{Uniformity over  $\Omega(\w)$ as $\w$ changes}
Next, we explain how the geometric constants for a base domain $\Omega$ are uniform as one varies the width of the necks for the associated chain domains $\Omega(\w)$.

Let $\Omega$ be a base domain with geometric constants $\rho^*, \kappa^*, \delta^*, \tau^*, w^*$.  
Then, for any collection of widths $\w=\{I_{ij,k}\}$, we have $|\Omega(\w)|>A^*$ and $L(\w)<L^*$, 
where 
$$|\Omega(\w)|:=\text{Area}(\Omega(\w)) \qquad \text{and}\qquad L(\w):=\text{Length}(\pa\Omega(\w)).$$
Therefore, we observe the following control on the geometry of $\Omega(\w)$:
\begin{enumerate}
    \item \emph{Isoperimetric ratio.} The isoperimetric ratio $L(\w)^2/|\Omega(\w)|$ of $\Omega(\w)$ is bounded  above by $\rho^*$.
     \item \emph{Normalized curvature.} The curvature of each smooth segment of $\pa\Omega(\w)$ is bounded above by $\kappa^*/L(\w)$.
    \item \emph{Vertex control.}
     The vertex control constant $\delta^*$ provides the same control on the boundary of $\Omega(\w)$ as for $\pa\Omega$, with $L^*$ replaced by $L(\w)$ and $[-1,1]$ replaced by $I_{ij,k}$.    In particular, there is a lower bound on the interior and exterior angle of each vertex of $\pa\Omega(\w)$ in terms of $\delta^*$, and the number of vertices of $\pa\Omega(\w)$ is bounded by $1/\delta^*$.

    \item  \emph{Normalized cut-distance.} 
    For $\eta\leq L(\w)\delta^*$ and all $\ell$, the cut-distance, given by
    \begin{align*}
    \inf_{u\in K_{\eta,\ell}}\sup_{\delta>0}\left\{\delta: \;\dist\Big(\gamma(u)+sn(u),\pa\D_{\ell}\Big)=s \; \text{for all} \;s\in[0,\delta] \Big.\right\},
\end{align*}
is bounded from below by $\tau^*\eta$. As above, $\gamma:K_{\eta,\ell}\to\R^2$ is a parameterization of $\pa \D_{\ell}$ with the parts of $\pa\D_{\ell}$ in the discs of radius $\eta$ centered at each vertex of $\pa\D_{\ell}$ excluded, for a union of intervals $K_{\eta,\ell}$, and $n(u)$ is the inward unit normal at $\gamma(u)\in\pa\D_{\ell}$.
    
    \item \emph{Neck-width.}
    The constant $w^*$ provides control on the ratio between the maximum and minimum width of the neck $ \Q_{ij,k}(\w)$ for all choices of $i,j,k$, in the sense that
    \begin{align}\label{e:max-min width2}
        & \min_s \frac{\min_{t\in I_{ij,k}}|\pa_t\G_{ij,k}(s,t)|}{\max_{t\in I_{ij,k}}|\pa_t\G_{ij,k}(s,t)|}\geq w^*, \qquad  w^* \leq |\pa_s\G_{ij,k}(s,t)| \leq 1/w^*, \\ 
        & \hspace{2cm} \min_{(s,t)}\frac{\left|{\rm det} [\pa_t\G_{ij,k}(s,t), \pa_s \G_{ij,k}(s,t)]\right|}{|\pa_t\G_{ij,k}(s,t)||\pa_s\G_{ij,k}(s,t)|}\geq w^*. \notag        
    \end{align}
\end{enumerate}
\begin{remark}[Uniform control in $\w$]\label{r:geom constants}
Our main result, Theorem \ref{thm:Main}, produces a bound  for a Courant sharp eigenvalue $\mu_m(\w)$ of the form  $|\Omega(\w)|\mu_m(\w)<C$, where $C$ only depends on $\rho^*, \kappa^*, \delta^*, \tau^*, w^*$.  The reason for this being possible is that the constant $C$ depends only on the geometric features of $\Omega(\w)$ described in points (1)-(5) above.
\end{remark}

For the domains $\Sw$ from Figure \ref{fig:toy}, with $0<w<1$, we can take 
\begin{align*}
\rho^* = 50, \quad \kappa^* = 0, \quad \delta^* = \tfrac{1}{40}, \quad \tau^* = 1, \quad w^* =1. 
\end{align*}

\begin{remark}[Notation] From now on, all constants $C^*$, $C_1^*$, etc., appearing may depend on $\rho^*, \kappa^*, \delta^*, \tau^*, w^*$, but will be independent of the collection of widths $\w=\{I_{ij,k}\}$ and any other geometric quantities involving $\Omega(\w)$. Constants $C$, $C_1$ without asterisks will be absolute constants, independent of $\Omega(\w)$.
\end{remark}

%%%%%%%%%%%%%%%%%%%%%%%%%%%%%%%%%%%%%%%%%
%%%%%%%%%%%%%%%%%%%%%%%%%%%%%%%%%%%%%%%%%
\section{Strategy and ingredients for the proof of Theorem \ref{thm:Main}}
%%%%%%%%%%%%%%%%%%%%%%%%%%%%%%%%%%%%%%%%%
%%%%%%%%%%%%%%%%%%%%%%%%%%%%%%%%%%%%%%%%%
In Section \ref{sec:ouline} below, we describe an outline for the proof of Theorem \ref{thm:Main}. The proof relies on introducing a partition of the chain domain into several regions and on classifying the nodal domains of an eigenfunction in terms of which regions they touch. The partition is introduced in Section \ref{sec:partition} and the classification of the domains is presented in Section \ref{sec:classification}  

%%%%%%%%%%%%%%%%%%%%%%%%%%%%%%%%%%%%%%%%%
\subsection{Outline of the proof of Theorem \ref{thm:Main}}\label{sec:ouline}
%%%%%%%%%%%%%%%%%%%%%%%%%%%%%%%%%%%%%%%%%
Let $\Omega(\w)$ be a chain domain as in Definition \ref{defn:domain},
and let $\{u_k(\w)\}_k$ and $\{\mu_k(\w)\}_k$ be the Neumann eigenfunctions and eigenvalues for $\Omega(\w)$ introduced in \eqref{e:ev}. In order to prove Theorem \ref{thm:Main}, we will use the same underlying strategy of proof as the various Pleijel-type results described in the introduction: we will establish upper and lower bounds on the Neumann counting function
\begin{align} \label{eqn:counting1}
\mathcal{N}^N_{\Omega(\w)}(\mu) := \#\{k:\mu_k(\w)<\mu\},
\end{align}
when $\mu = \mu_m(\w)$ is a Courant sharp eigenvalue; that is, $\nu(u_m(\w))=m$, where we continue to write  $\nu(u_m(\w))$ for the number of nodal domains of $u_m(\w)$.

In this Courant sharp case, $\mu_k(\w) \neq \mu_m(\w)$ for $k<m$, and so $\mathcal{N}^N_{\Omega(\w)}(\mu_m(\w))=m-1.$  We therefore have
\begin{align} \label{eqn:counting2}
    \mathcal{N}^N_{\Omega(\w)}(\mu_m(\w))+1= m= \nu(u_m(\w)).
\end{align}
We find an upper bound on $\nu(u_m(\w))$ of the form $C_1^*\mu_m(\w)|\Omega(\w)|+O\left((\mu_m(\w)|\Omega(w)|)^{3/4}\right)$ from deriving lower bounds on the area of a nodal domain of $u_m(\w)$ (See Section \ref{sec:UB}). At the same time, a Weyl law with an explicit bound on the remainder then gives a lower bound on $ \mathcal{N}^N_{\Omega(\w)}(\mu_m(\w))$ of the form $C_2^*\mu_m(\w)|\Omega(\w)|+O\left((\mu_m(\w)|\Omega(w)|)^{3/4}\right)$ (See Section \ref{sec:LB}). The constants $C_1^*, C_2^*$ are explicit enough that we can argue $C_2^*>C_1^*$ and hence derive from \eqref{eqn:counting2} that $$\mu_m(\w)|\Omega(\w)|<C^*$$ as claimed. 
The detailed proof of Theorem \ref{thm:Main} is done  in Section \ref{sec:end}. The upper bound on $\nu(u_m(\w))$ will hold for any (not necessarily Courant sharp) Neumann eigenfunction, and will imply the limit in \eqref{e:main ub}.

%%%%%%%%%%
\subsubsection{Strategy for obtaining the upper bound on $\nu(u_m(\w))$.}\label{sec:UB} 
%%%%%%%%%%
To obtain lower bounds on the area of a nodal domain, we will adapt the strategy from \cite{gittins-lena2020} and \cite{lena-pleijel} by splitting each nodal domain $D$ of $u_m$ into four different categories, depending on where the $L^2(D)$-mass of $u_m$ is concentrated. Roughly speaking, we will estimate the number of nodal domains of $u_m$ in each of the following four cases:
\begin{enumerate}
    \item[i)] the majority of the mass is concentrated away from the necks $ \Q_{ij,k}(\w)$ and the boundaries of the domains $\D_\ell$;
    \item[ii)] some of the mass is concentrated away from the necks $ \Q_{ij,k}(\w)$ and near the smooth parts of the boundary of the domains $\D_\ell$;
    \item[iii)] some of the mass is concentrated near a vertex of $\D_\ell$ or the ends of the necks $ \Q_{ij,k}(\w)$;
    \item[iv)] some of the mass is concentrated in the necks $ \Q_{ij,k}(\w)$.
\end{enumerate}
Using the established techniques from \cite{gittins-lena2020} and \cite{lena-pleijel}, cases i) and ii) can be handled using the Faber-Krahn Theorem together with a reflection argument across the smooth part of the boundary of $\Omega(\w)$. We will recall this argument and define the above partitioning of the nodal domains in Sections \ref{sec:bulk} and  \ref{sec:partition} respectively.

The novelty of our work lies in the remaining cases. For case iii), in Section \ref{sec:corner} we will exploit properties of a Neumann eigenfunction near a corner in order to bound the number of such nodal domains near each vertex in $\Omega(\w)$. Finally for case iv), in Section \ref{sec:neck} we will prove a non-sharp version of the Faber-Krahn Theorem for thin cylinders in order to obtain a lower bound on the area of nodal domains contained in the neck. In particular, for sufficiently large eigenvalues $\mu_m(\w)$, the number of nodal domains in cases ii), iii), and iv) will be small compared to case i). 

%%%%%%%
\subsubsection{Strategy for obtaining the lower bound on  $\mathcal{N}^N_{\Omega(\w)}(\mu_m(\w))$.}\label{sec:LB}
%%%%%%%%
For the lower bound on \eqref{eqn:counting2}, we will use a Weyl law with an explicit bound on the remainder. This comes from a Weyl remainder estimate given in \cite{vdberg-gittins2016} involving $M_{\Omega(\w)}(\delta)$, as defined in \eqref{e:M}. We then bound $M_{\Omega(\w)}(\delta)$ in terms of $\delta$ and the five geometric constants $\rho^*, \kappa^*, \delta^*, \tau^*, w^*$. This estimate will be given in Section \ref{sec:Weyl}.

%%%%%%%%%%%%%%%%%%%%%%%%%%%%%%%%%%%%%%%%%
\subsection{Partition of the chain domains} \label{sec:partition}
%%%%%%%%%%%%%%%%%%%%%%%%%%%%%%%%%%%%%%%%%

Let $\Omega(\w)$ be a chain domain as in Definition \ref{defn:domain}. As outlined in Section \ref{sec:UB}, we estimate the number of nodal domains by splitting our study into a series of cases. This is achieved by partitioning the domain. In this section, we define the partition, establish its required properties, and use it to define bulk, boundary, corner, and neck nodal domains.

 Given $\delta>0$, we partition $\Omega(\w)$ into $\bigcup_{j=0}^{4}\Omega_j^{\delta}(\w)$. Roughly speaking, 
 \begin{enumerate}
\item[-]  $\Omega_0^{\delta}(\w)$ is the part of $\Omega(\w)$ a distance $\delta$ away from the boundary;
\item[-] $\Omega_1^{\delta}(\w)$ is a $\delta$-neighborhood of the smooth part of $\pa\Omega(\w)$;
\item[-] $\Omega_2^{\delta}(\w)$ is a $\delta$-neighborhood of the vertices of $\D_{\ell}$.
 \end{enumerate}

In what follows, we continue to write $w_{ij,k}$ for the minimum neck widths introduced in \eqref{e:ws}.
 
 When $\delta$ is small compared to the minimum width $w_{ij,k}$ of the neck $ \Q_{ij,k}(\w)$, then $\Omega_2^{\delta}(\w)$ also contains a $\delta$-neighborhood of the vertices where the neck $ \Q_{ij,k}(\w)$ is joined to the domains $\D_{i}$, $\D_{j}$. However, when $\delta$ is large compared to the minimum width $w_{ij,k}$ of the neck $ \Q_{ij,k}(\w)$, then $\Omega_4^{\delta}(\w)$ contains a $\delta$-neighborhood of these vertices, and $\Omega_3^{\delta}(\w)$ then contains the rest of the neck $ \Q_{ij,k}(\w)$.
 
Throughout, we will work with $\delta>0$ satisfying
 \begin{align} \label{eqn:delta-partition}
     \delta \leq \min\left\{\frac{1}{20}L(\w)\delta^*,\frac{L(\w)}{\kappa^*\tau^*} \right\},
 \end{align}
 with $\delta^*$, $\kappa^*$, $\tau^*$ the vertex control, normalized curvature, and normalized cut-distance constants of $\Omega(\w)$. Recall that $L(\w)\delta^*$ gives a lower bound on the distance between vertices of $\pa\Omega(\w)$, and so this bound on $\delta$ guarantees that we can cleanly separate a $\delta$ neighborhood of each vertex of $\D_\ell$, and the rest of its boundary. The upper bound on $\delta$ of $L(\w)/(\kappa^*\tau^*)$ will also ensure that, after excluding a disc of radius $\delta$ centered at each vertex of $\D_\ell$, the cut-distance of the remaining part of $\pa \D_\ell$ is bounded from below by $\delta$. This will be important because it will allow us to apply a diffeomorphism to straighten this resulting part of the boundary.
 
We proceed to give the precise definition of the partition. 

 \begin{defn}[$\delta$-partition of a chain domain] \label{defn:partition}
Let $\Omega(\w)$ be a chain domain as in Definition \ref{defn:domain} and $\delta>0$ satisfy \eqref{eqn:delta-partition}. Then,  
the $\delta$-partition for $\Omega(\w)$ is defined as  
$$
\Omega(\w) = \bigcup_{j=0}^{4}\Omega_j^{\delta}(\w),
$$
where the following holds:
\begin{enumerate}
\item[1)] $\Omega_2^{\delta}(\w)$ contains a disc of radius $\delta$ around each vertex of $\D_\ell$ for all $1\leq \ell\leq M$.

\item[2)] For each $1\leq i\leq j\leq M$, $1\leq k \leq K_{ij}$ one of the following  holds.

\begin{enumerate}
\item[i)] {If $w_{ij,k}>4\delta$}, then 
\begin{itemize}
\item $\Omega_2^{\delta}(\w)$ contains the disc of radius $\delta$ around each vertex formed by the neck $ \Q_{ij,k}(\w)$ and the domains $\D_i$, $\D_j$;
\item $\Omega_{1}^{\delta}(\w)$ contains the part of $ \Q_{ij,k}(\w)$ near its boundary, given by
$$
    \Big\{x\in  \Q_{ij,k}(\w)\,:\,\dist(x,\pa  \Q_{ij,k}(\w))<\tfrac{3}{4}\tau^*\delta\Big \}\backslash\Omega_2^{\delta}(\w).
$$
\end{itemize}

\item[ii)] If $w_{ij,k}\leq 4\delta$, then 
\begin{itemize}
\item $\Omega_{4}^{\delta}(\w)$ contains the $\delta$-neighborhood of the ends of the neck $\G_{ij,k}(\{0\}\times I_{ij,k})$ and $\G_{ij,k}(\{L_{ij,k}\}\times I_{ij,k})$;
\item $\Omega_3^{\delta}(\w)$ then contains the rest of the neck, $ \Q_{ij,k}(\w)\backslash\Omega_{4}^{\delta}(\w)$.
\end{itemize}

\end{enumerate}

\item[3)]$\Omega_1^{\delta}(\w)$ includes 
\begin{align*}
    \Big\{x\in \bigcup_{\ell=1}^M \D_\ell:\,\dist(x,\pa\Omega(\w))<\tfrac{3}{4}\tau^*\delta\Big\}\backslash\left(\Omega_2^{\delta}(\w)\cup \Omega_4^{\delta}(\w)\right).
\end{align*}
\item[4)] $\Omega_0^{\delta}(\w)$ is defined by
\begin{align*}
    \Omega_0^{\delta}(\w) & =  \Omega(\w)\backslash \left(\Omega_1^{\delta}(\w) \cup\Omega_2^{\delta}(\w) \cup \Omega_3^{\delta}(\w) \cup \Omega_4^{\delta}(\w)\right).
\end{align*}

\end{enumerate}
\end{defn}
See Figure \ref{fig:partition} for an example of the possible $\delta$-partitions for the model domain $\Sw$ introduced in Figure \ref{fig:toy}, depending on the relative size of  $\delta$, and the neck width $w$.

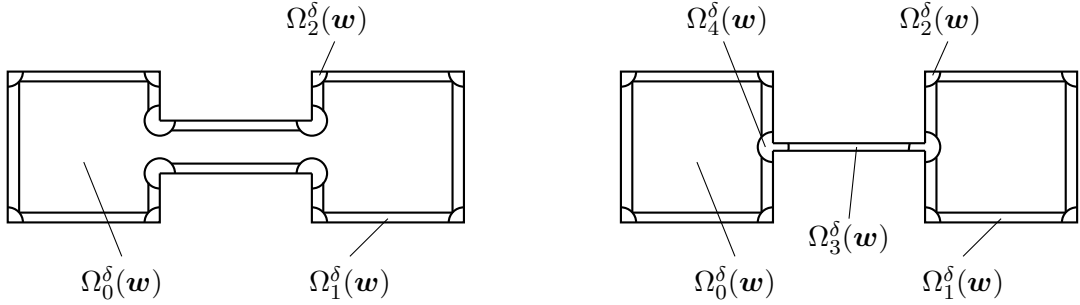
\begin{figure}[h]
    \centering
    
\begin{tikzpicture}
\draw[thick] (-1,0.35) -- (1,0.35);
\draw[thick] (-1,-0.35) -- (1,-0.35);
\draw[thick] (1,0.35) -- (1,1) -- (3,1) -- (3,-1) -- (1,-1) -- (1,-0.35);
\draw[thick] (-1,0.35) -- (-1,1) -- (-3,1) -- (-3,-1) -- (-1,-1) -- (-1,-0.35);

\draw[thick] (-3,0.8) arc (-90:0:0.2);
\draw[thick] (-2.8,-1) arc (0:90:0.2);
\draw[thick] (-1,-0.8) arc (90:180:0.2);
\draw[thick] (-1.2,1) arc (180:270:0.2);
\draw[thick] (1,0.8) arc (-90:0:0.2);
\draw[thick] (1.2,-1) arc (0:90:0.2);
\draw[thick] (3,-0.8) arc (90:180:0.2);
\draw[thick] (2.8,1) arc (180:270:0.2);
\draw[thick] (-1,0.55) arc (90:360:0.2);
\draw[thick] (-0.8,-0.35) arc (0:270:0.2);
\draw[thick] (0.8,0.35) arc (-180:90:0.2);
\draw[thick] (1,-0.55) arc (-90:180:0.2);
\draw[thick] (2.85,0.87) -- (2.85,-0.87);
\draw[thick] (-2.85,0.87) -- (-2.85,-0.87);
\draw[thick] (1.15,0.87) -- (2.85,0.87);
\draw[thick] (1.15,-0.87) -- (2.85,-0.87);
\draw[thick] (-1.15,0.87) -- (-2.85,0.87);
\draw[thick] (-1.15,-0.87) -- (-2.85,-0.87);
\draw[thick] (1.15,0.87) -- (1.15,0.47);
\draw[thick] (1.15,-0.87) -- (1.15,-0.47);
\draw[thick] (-1.15,0.87) -- (-1.15,0.47);
\draw[thick] (-1.15,-0.87) -- (-1.15,-0.47);
\draw[thick] (-0.85,0.22) -- (0.85,0.22);
\draw[thick] (-0.85,-0.22) -- (0.85,-0.22);
\node at (-1.5,-1.8) {$\Omega_0^{\delta}(\w)$};
\node at (1.5,-1.8) {$\Omega_1^{\delta}(\w)$};
\node at (1.2,1.7) {$\Omega_2^{\delta}(\w)$};
\draw (-1.5,-1.45) -- (-2,-0.2);
\draw (1.5,-1.45) -- (2,-0.95);
\draw (1.3,1.4) -- (1.1,0.95);
\end{tikzpicture}    
\hspace{0.7in}    
\begin{tikzpicture}
\draw[thick] (-1,0.05) -- (1,0.05);
\draw[thick] (-1,-0.05) -- (1,-0.05);
\draw[thick] (1,0.05) -- (1,1) -- (3,1) -- (3,-1) -- (1,-1) -- (1,-0.05);
\draw[thick] (-1,0.05) -- (-1,1) -- (-3,1) -- (-3,-1) -- (-1,-1) -- (-1,-0.05);

\draw[thick] (-3,0.8) arc (-90:0:0.2);
\draw[thick] (-2.8,-1) arc (0:90:0.2);
\draw[thick] (-1,-0.8) arc (90:180:0.2);
\draw[thick] (-1.2,1) arc (180:270:0.2);
\draw[thick] (1,0.8) arc (-90:0:0.2);
\draw[thick] (1.2,-1) arc (0:90:0.2);
\draw[thick] (3,-0.8) arc (90:180:0.2);
\draw[thick] (2.8,1) arc (180:270:0.2);
\draw[thick] (-1,0.2) arc (90:270:0.2);
\draw[thick] (1,-0.2) arc (-90:90:0.2);
\draw[thick] (-0.8,-0.05) arc (-15:15:0.2);
\draw[thick] (0.8,0.05) arc (165:180:0.4);
\draw[thick] (2.85,0.87) -- (2.85,-0.87);
\draw[thick] (-2.85,0.87) -- (-2.85,-0.87);
\draw[thick] (1.15,0.87) -- (2.85,0.87);
\draw[thick] (1.15,-0.87) -- (2.85,-0.87);
\draw[thick] (-1.15,0.87) -- (-2.85,0.87);
\draw[thick] (-1.15,-0.87) -- (-2.85,-0.87);
\draw[thick] (1.15,0.87) -- (1.15,0.14);
\draw[thick] (1.15,-0.87) -- (1.15,-0.14);
\draw[thick] (-1.15,0.87) -- (-1.15,0.14);
\draw[thick] (-1.15,-0.87) -- (-1.15,-0.14);
\node at (-1.5,-1.8) {$\Omega_0^{\delta}(\w)$};
\node at (1.5,-1.8) {$\Omega_1^{\delta}(\w)$};
\node at (1.2,1.7) {$\Omega_2^{\delta}(\w)$};
\node at (0,-1.2) {$\Omega_3^{\delta}(\w)$};
\node at (-1.6,1.7) {$\Omega_4^{\delta}(\w)$};
\draw (-1.5,-1.45) -- (-2,-0.2);
\draw (1.5,-1.45) -- (2,-0.95);
\draw (0,-0.95) -- (0.1,0);
\draw (-1.65,1.4) -- (-1.1,0);
\draw (1.3,1.4) -- (1.1,0.95);
\end{tikzpicture}
    \caption{The two $\delta$-partitions of the chain domain $\Omega(\w)$ introduced in Figure \ref{fig:toy}, depending on the relative size of $w$ and $\delta$.} 
    \label{fig:partition}
\end{figure}

We will later work with a partition of unity adapted to the $\delta$-partition of a chain domain. Before introducing it, we present the reader with a lemma that allows us to straighten the boundary of the domain. This lemma will allow us to define the partition of unity and will also be used in Section~\ref{sec:bulk}.

\begin{lem} \label{lem:boundary1}
Let $\Omega(\w)$ be a chain domain. Fix $\eta>0$ with $\eta \leq L(\w)\min\left\{\delta^*,\tfrac{1}{\kappa^*\tau^*}\right\}$. Given a side of $\D_{\ell}$, denote $b_{\eta}$ to be the part of this side a distance at least $\eta$ from the vertices of $\D_{\ell}$. Letting $\{\gamma(s):s\in I
\}$ be a parameterization of $b_{\eta}$, and $n(s)$ to be the unit inward normal to $\pa \D_{\ell}$ at $\gamma(s)$, a neighborhood of $b_{\eta}$ in $\D_{\ell}$ can be straightened in the following sense.

The function
\begin{align*}
    F:I\times[0,\tfrac{3}{4}\tau^*\eta]  \to \D_{\ell}, \qquad 
    (s,t)\mapsto (x,y) = \gamma(s) + tn(s),
\end{align*}
is a diffeomorphism onto its image. Moreover, the Jacobian of this change of variables is bounded from above and below by
\begin{align*}
    1 + \tfrac{3}{4}L(\w)^{-1}\tau^*\kappa^*\eta \leq \tfrac{7}{4}, \qquad     1 - \tfrac{3}{4}L(\w)^{-1}\tau^*\kappa^*\eta \geq \tfrac{1}{4}
\end{align*}
respectively.
\end{lem}

\begin{proof} By the definition of the normalized cut-distance constant $\tau^*$, since $\eta \leq L(\w)\delta^*$, the cut-distance of $b_{\eta}$ is bounded from below by $\tau^*\eta$. The proof of this lemma then follows in an identical way to the argument in Section 3 of \cite{gittins-lena2020}, and so we omit the details.
\end{proof}
\begin{remark} \label{rem:neck-straighten}
For each neck $\Q_{ij,k}(\w)$, 
%assuming without loss of generality that the interval $I_{ij,k}$ is centered at $0$, 
we define the diffeomorphism $F_{ij,k}$ by
\begin{align*}
    F_{ij,k}:[0,L_{ij,k}]&\times (-w_{ij,k},w_{ij,k}) \to \Q_{ij,k}(\w) \\
    F_{ij,k}(s,t) & = \G_{ij,k}\left(s,t_1 + (t+w_{ij,k}) |I_{ij,k}|/(2w_{ij,k})\Big.\right),
\end{align*}
where we recall $t_1 < 0$ defines an endpoint of the interval $I_{ij,k}$.  By the definition of $w_{ij,k}$ in \eqref{e:ws} and the neck-width constant $w^*$, $|\pa_t\G_{ij,k}(s,t)|$ is bounded from above and below by $w_{ij,k}/|I_{ij,k}|$ multiplied by constants depending only on $w^*$. Therefore, the Jacobian of $F_{ij,k}$ is bounded in terms of the neck-width constant $w^*$, and so in particular this can be used to straighten the top and bottom boundaries of $\Q_{ij,k}(\w)$.
\end{remark}

We proceed to introduce the partition of unity associated to the $\delta$-partition of a chain domain.
\begin{lem} \label{lem:partition}
Let $\Omega(\w)$ be a chain domain.
There exists a constant $C^*>0$ such the following holds. For each $\delta>0$ satisfying \eqref{eqn:delta-partition} there exist  smooth functions $\{\chi_j^{\delta}\}_{j=0}^4$ associated to the $\delta$-partition  $\{\Omega_j^{\delta}(\w)\}_{j=0}^4$ of $\Omega(\w)$ such that  
$$
\sum_{j=0}^{4} \big(\chi_j^{\delta}\big)^2 \equiv 1 \quad \text{on}\;\;\Omega(\w),
$$
\begin{enumerate}

    \item  $\chi_j^{\delta}\equiv 1$ on $\Omega_j^{\delta}(\w)$ for $j=0,3$,
    
    \item  $\chi_1^{\delta}\equiv 1$ on $\Omega_1^{\delta/2}(\w)$,
    
    \item $\chi_j^{\delta}\equiv 1$ on  $\Omega_j^{\delta/4}(\w)$ for $j=2,4$,
    
    \item  $\chi_j^{\delta} \in H^1 (\Omega(\w)) $ with $\big|\nabla \chi_j^{\delta}\big| \leq C^*\delta^{-1}$ \; a.e. on $\Omega(\w)$ for $j=0,\dots,4$.
    
\end{enumerate}
\end{lem}

\begin{proof}
This result follows from the definition of the $\delta$-partition, and the straightening of the smooth parts of the boundary given in Lemma \ref{lem:boundary1} below:

First, the cut-off function $\chi_2^\delta$ is straightforward to define as radial functions, centered at the points used to define $\Omega_2^\delta(\w)$, and on a length scale comparable to $\delta$.

To define $\chi_1^{\delta}$, we use the upper bound on $\delta$ from \eqref{eqn:delta-partition}. This allows us to apply Lemma \ref{lem:boundary1} above with $\eta = \delta$. Using the diffeomorphism $F$ from this lemma straightens each part of the side of $\D_{\ell}$ in $\Omega_1^\delta(\w)$, via a change of variables with a bounded Jacobian. We can also use the diffeomorphism from Remark \ref{rem:neck-straighten} to straighten the top and bottom boundaries of $\Q_{ij,k}(\w)$ in the case where $w_{ij,k}>4\delta$. It is then straightforward to define the cut-off function $\chi_1^\delta$ with the desired properties.

This straightening of the sides of $\D_{\ell}$ and the necks $\Q_{ij,k}(\w)$ also allow for the function $\chi_4^{\delta}$ to be defined on a length scale comparable to $\delta$, with the required properties. Note that since the neck and sides are straightened using two  different diffeomorphisms, we can use them to define a continuous cut-off function $\chi_4^{\delta}$ in $H^1(\Omega(\w))$, with a possibly discontinuous derivative at the intersection of their supports. This is the reason for the almost everywhere nature of the pointwise bound stated in $(4)$ above.

The function $\chi_3^{\delta}$ is then defined with support in the necks $ \Q_{ij,k}(\w)$ with $w_{ij,k}\leq 4\delta$, and so that 
\begin{align*}
    (\chi_3^\delta)^2+(\chi_4^\delta)^2 \equiv1 \text{ on }  \Q_{ij,k}(\w).
\end{align*}
Finally, $\chi_0^{\delta}$ is defined so that $\sum_{j=0}^{4}(\chi_j^{\delta})^2 \equiv 1$ on $\Omega(\w)$ as required.
\end{proof}

%\begin{remark}
%{\color{blue} Once we are happy with the above proof and remark, we can omit this remark.} Since the top and bottom boundaries of each neck $ \Q_{ij,k}(\w)$ are graphs with slopes $\leq w_0$, \ya{\cs discuss\cs} these boundaries can also be separately straightened with a bounded Jacobian.    {\color{magenta} Happy to remove this remark once the above proof is settled.}
%\end{remark}

%%%%%%%%%%%%%%%%%%%%%%%%%%%%%%%%%%%%%%%%%
\subsection{Classification of the nodal domains}\label{sec:classification}
%%%%%%%%%%%%%%%%%%%%%%%%%%%%%%%%%%%%%%%%%
Let $u$ be a Neumann eigenfunction of $\Omega(\w)$, not necessarily Courant sharp, with eigenvalue $\mu$, and let $D$ be one of its nodal domains. Fix $\epsilon, \delta > 0$.  We will use $\epsilon$ to measure the extent to which the $L^2(D)$-mass of $u$ is concentrated away from the boundary. As in Definition \ref{defn:partition}, $\delta$ is used to partition the neighborhoods of different parts of the boundary of $\Omega(\w)$. We will eventually choose $\delta$ in terms of the area of $\Omega(\w)$ and the eigenvalue $\mu$. We follow a similar framework to that in \cite{lena2015courant} and \cite{gittins-lena2020}, but with more regions.

Let $\{\chi_j^{\delta}\}_{j=0}^4$ be the associated partition of unity to the $\delta$-partition  $\{\Omega_j^{\delta}(\w)\}_{j=0}^4$ of $\Omega(\w)$ (see Lemma \ref{lem:partition}).
We decompose
\begin{equation*}
    %\label{u:decomp}
    u = \sum_{j=0}^4u_j ,\qquad u_j := \chi_j^\delta u.
\end{equation*}
So, $u_0$ is localized in the interior, $u_1$ is localized near the smooth part of the boundary, $u_2$ is localized near corners of $\D_\ell$, $u_3$ localized near necks $\Q_{ij,k}(\w)$, and $u_4$ is localized near where the necks are joined to the domains $\D_\ell$.  %**normalize $u$ so that $\int_D u^2=1$ below?** 

For $\epsilon>0$ and $\delta > 0$, 
the collection of bulk, boundary, corner, and neck nodal domains will be denoted by 
\begin{equation}\label{e:dom}
\mathcal{V}^{\delta}_0(\eps; u), \;\mathcal{V}^{\delta}_1(\eps;u), \;\mathcal{V}^{\delta}_2(\eps;u),\; \mathcal{V}^{\delta}_3(\eps;u),
\end{equation}
respectively, and defined in the following way: 
\begin{itemize}
\item \textit{bulk nodal domains:} $D \in \mathcal{V}^{\delta}_0(\eps; u)$ if 
\begin{equation}\label{e:bulk nd}
\|u_0\|_{L^2(D)}^2\geq (1-\epsilon) \|u\|_{L^2(D)}^2.
\end{equation}

\item \textit{boundary nodal domains:}  $D \in \mathcal{V}^{\delta}_1(\eps; u)$ if 
\begin{equation}\label{e:boundary nd}
\|u_1\|_{L^2(D)}^2\geq \tfrac{1}{4}\varepsilon  \|u\|_{L^2(D)}^2.
\end{equation}

\item  \textit{corner nodal domains:} $D \in \mathcal{V}^{\delta}_2(\eps; u)$ if 
\begin{equation}\label{e:corner nd}
\|u_2\|_{L^2(D)}^2\geq \tfrac{1}{4}\varepsilon  \|u\|_{L^2(D)}^2 \quad
\text{or} \quad
\|u_4\|_{L^2(D\cap \cup_\ell \D_\ell)}^2 \geq \tfrac{1}{8}\eps\|u\|_{L^2(D)}^2.
\end{equation}

\item  \textit{neck nodal domains:} $D \in \mathcal{V}^{\delta}_3(\eps; u)$ if 
\begin{equation}\label{e:neck nd}
\|u_3\|_{L^2(D)}^2\geq \tfrac{1}{4}\varepsilon  \|u\|_{L^2(D)}^2 \quad
\text{or} \quad
\|u_4\|_{L^2(D\cap \cup  \Q_{ij,k}(\w))}^2 \geq \tfrac{1}{8}\eps\|u\|_{L^2(D)}^2.
\end{equation}
\end{itemize}

For $j=0,1,2,3$ we write
\begin{equation}\label{e:dom number}
\nu^{\delta}_j(\eps; u):= \# \mathcal{V}^{\delta}_j(\eps;u)
\end{equation}
for the number of domains in each class.

Note that a given nodal domain may fall into more than one of these categories, but we always have that the total number of nodal domains for $u$ is bounded by
\begin{equation}\label{e: nd bd}
\nu(u) \leq \sum_{j=0}^3\nu^{\delta}_j(\eps; u).
\end{equation}

%%%%%%%%%%%%%%%%%%%%%%%%%%%%%%%%%%%%%%%%%
\subsubsection{Green's formula for nodal domains}
%%%%%%%%%%%%%%%%%%%%%%%%%%%%%%%%%%%%%%%%%
A key ingredient in obtaining an upper bound on the number of each type of nodal domain will be the Faber-Krahn theorem applied to regions where we have an upper bound on its first Dirichlet eigenvalue. In order to obtain this upper bound, we will need that the following Green's formula holds for each nodal domain.
\begin{lem} \label{lem:Green}
Any nodal domain $D$ of a Neumann eigenfunction $u$ of\, $\Omega(\w)$, with eigenvalue $\mu$, satisfies
\begin{align*}
    \int_{D}\left|\nabla u\right|^2 = \mu \int_{D}u^2.
\end{align*}
\end{lem}
\begin{proof}
To prove the lemma, we will use the following version of Green's identity, given in \cite[Lemma 1.5.3.8]{grisvard2011elliptic}: Let $\Omega$ be a bounded, open set in $\mathbb{R}^2$, with boundary $\pa\Omega$ given by a $C^{1,1}$ curvilinear polygon. Then, for $v_1\in H^2(\Omega)$, $v_2\in H^1(\Omega)$, 
\begin{equation} \label{eqn:Green1}
    \int_{\Omega}(\Delta v_1)v_2 = -\int_{\Omega}\nabla v_1\cdot \nabla v_2 + \int_{\pa\Omega}\frac{\pa v_1}{\pa \nu}v_2.
\end{equation}
In the case where the nodal domain $D$ does not contain a corner of $\Omega(\w)$, then the eigenfunction $u$ is smooth in $D$. Moreover, since $\Omega(\w)$ is planar, the boundary of $D$ is piecewise $C^1$ and meets at equal angles in the interior and on the boundary. Therefore, we have sufficient regularity to apply \eqref{eqn:Green1} with $v_1=v_2 = u$ and $\Omega = D$. Since $\Delta u = -\mu u$, and $u$ satisfies Neumann boundary conditions on $\pa\Omega$, together with Dirichlet boundary conditions on the rest of $\pa D$, this gives the equality in the statement of the lemma.

To handle the case where $D$ contains a vertex of $\Omega(\w)$, we need to use the regularity of the Neumann eigenfunction $u$ at the vertex. Let $\theta$ be the interior angle of a given vertex. Then, by the theorem in Section 1 of \cite{wigley1970}, $u$ satisfies these estimates in a neighborhood of the vertex:
\begin{enumerate}
    \item[i)] if $\theta <\pi$, then $u$ is $C^1$ in a neighborhood of the vertex;
    \item[ii)] if $\theta>\pi$, then $u$ is H\"older continuous, with exponent $\pi/\theta$, in a neighborhood of the vertex, and $\limsup_{r\to0}r^{1-\pi/\theta}|\nabla u|<\infty$, where $r$ is the distance to the vertex. 
\end{enumerate}
In particular, since $\theta<2\pi$, this ensures that $\nabla u \in L^4$.

We now define $D_k$ to be the set formed by intersecting $D$ with discs of radius $\eps_k>0$ centered at the corners of $\Omega(\w)$. We can choose the sequence $\eps_k$, with $\lim_{k\to\infty}\eps_k = 0$, so that the domains $D_k$ are $C^{1,1}$-curvilinear polygons. Moreover, $u$ is smooth in $D_k$, and so applying \eqref{eqn:Green1}, with $v_1=v_2=u$, we obtain
\begin{align*}
        -\int_{D_k}\mu|u|^2 = -\int_{D_k}|\nabla u|^2 + \int_{\pa D_k}\frac{\pa u}{\pa \nu}u.
\end{align*}
By the boundary conditions satisfied by $u$, the only contribution to the boundary integral is from arcs of the circles of radius $\eps_k$ centered at the corners of $\Omega(\w)$. By the regularity of $u$ and $\nabla u$, these contributions go to $0$ as $\eps_k$ tends to $0$. Therefore, taking the limit $\eps_k\to0$, we obtain the equality in the statement of the lemma, and this completes the proof.
\end{proof}
\begin{remark}
\label{rmk:lip}
As we are working in two dimensions, the nodal domain $D$ has a Lipschitz boundary, and meets the smooth part of the boundary of $\Omega(\w)$ at non-zero angles.  For smooth components of the boundary this is proven in Theorem $2.3$ in \cite{helffer2009nodal}, and at corners it is proven in Theorem $2.6$ in \cite{helffer2009nodal}.  We will also assume that $D\cap\{u_j\neq0\}$ has these same properties for $0\leq j \leq 4$, which can be achieved by replacing $\delta$ by a sequence $\delta_n\to\delta$ if necessary.
\end{remark}

%%%%%%%%%%%%%%%%%%%%%%%%%%%%%%%%%%%%%%%%%
%%%%%%%%%%%%%%%%%%%%%%%%%%%%%%%%%%%%%%%%%
\section{Estimates on bulk and boundary nodal domains} \label{sec:bulk}
%%%%%%%%%%%%%%%%%%%%%%%%%%%%%%%%%%%%%%%%%
%%%%%%%%%%%%%%%%%%%%%%%%%%%%%%%%%%%%%%%%%

In this section, we bound the number of bulk and boundary nodal domains. See Section \ref{sec:classification} for the classification and \eqref{e:dom number} for the notation.

\subsection{Number of bulk nodal domains}
We recall that for $\delta, \epsilon$  fixed and $u$ being a Neumann eigenfunction for the chain domain, a bulk nodal domain, $D\in \mathcal{V}^\delta_0(\epsilon, u)$,  is one for which \eqref{e:bulk nd} holds. As introduced in \eqref{e:dom number}, we continue to denote the corresponding number of bulk nodal domains by $ \nu^{\delta}_0(\eps;u)$.  Our main result is the following.

\begin{prop}[Number of bulk nodal domains] \label{prop:bulk}
Let $\Omega(\w)$ be a chain domain. There exists a constant $C_0^*>0$ such that for all $0<\eps<\tfrac{1}{2}$ and $0<\beta<\tfrac{1}{2}$ the following holds. If $u$ is a Neumann eigenfunction of $\Omega(\w)$ with eigenvalue $\mu$ and $|\Omega(\w)|\mu \geq (C_0^*)^{1/\beta}$, then
\begin{align}    \label{nubd:int0}
    \nu^{\delta}_0(\eps;u) \leq \frac{1}{\pi\lambda_1(\mathbb{D})}\left(\frac{1+\eps}{1-\eps}(|\Omega(\w)|\mu) + \frac{1+\tfrac{1}{\eps}}{1-\eps}C_0^*(|\Omega(\w)|\mu)^{2\beta} \right)
\end{align}
for $\delta=|\Omega(\w)|^{1/2-\beta}\mu^{-\beta}$.
\end{prop}
\begin{proof}
To bound the number of bulk nodal domains of $u$, we follow exactly the argument to derive  \cite[equation (16)]{gittins-lena2020} and  \cite[equation (4)]{lena-pleijel}, using the properties of the $\delta$-partition from Lemma \ref{lem:partition}. This involves applying the Faber-Krahn inequality to each bulk domain, where we use $u_0=\chi_0^{\delta}u$ as a test function to obtain an upper bound on its first Dirichlet eigenvalue. This argument gives the estimate on the area of a bulk nodal domain, $D \in \mathcal{V}^\delta_0(\epsilon, u)$, of the form
\begin{equation}
    \label{FK:Dbd:int}
   \pi \lambda_1(\mathbb{D}) |D|^{-1} \leq \frac{1+\eps}{1-\epsilon} \,\frac{  \int_D | \nabla u|^2 dx}{ \int_D |u|^2 dx} + \frac{(1+\tfrac{1}{\eps})C_0^*}{(1-\eps)\delta^2}.
\end{equation}
 Using Lemma \ref{lem:Green}, this gives an upper bound on the number of such bulk nodal domains of 
\begin{equation}
    \label{nubd:int}
    \nu^{\delta}_0(\eps;u) \leq \frac{1}{\pi\lambda_1(\mathbb{D})} \left(  \frac{1+\epsilon}{1-\epsilon}  \mu + \frac{1 + \frac{1}{\epsilon}}{1 - \epsilon} C_0^* \delta^{-2} \right)|\Omega(\w)|. 
\end{equation}
The result would then hold once we show that the upper bound on $\delta$ from  \eqref{eqn:delta-partition} holds.
To see this, note that the isoperimetric inequality yields $|\Omega(\w)|^{-1/2}L(\w) \geq 2\pi^{1/2}$. Thus, setting 
$$c_0^*:=2\pi^{1/2} \min\left\{\tfrac{\delta_*}{20},\tfrac{1}{\kappa^*\tau^*} \right\},$$ 
we have 
\begin{align} \label{eqn:mu-partition}
\delta |\Omega(\w)|^{-1/2}= (|\Omega(\w)|\mu)^{-\beta}\leq {c_0^*} \leq |\Omega(\w)|^{-1/2}  L(\w)\min\left\{\tfrac{\delta_*}{20},\tfrac{1}{\kappa^*\tau^*} \right\},
\end{align}
as needed, provided $|\Omega(\w)|\mu\geq (c_0^*)^{-1/\beta}$.
\end{proof}

We will later choose $\eps>0$ small so that $\tfrac{1+\eps}{1-\eps}$ is sufficiently close to $1$, and $\beta=3/8$ so that the lower bound on $|\Omega(\w)|\mu$ can be written as simply $C_0^*$.

The first term in the estimate in Proposition \ref{prop:bulk} will be the leading order term in the count of the number of nodal domains. It is therefore important that the second term is sub-linear in $|\Omega(\w)|\mu$ for $0<\beta<\tfrac{1}{2}$.

%%%%%%%%%%%%%%%%%%%%%%%%%%%%%%%%%%%%%%%%%
%%%%%%%%%%%%%%%%%%%%%%%%%%%%%%%%%%%%%%%%%
\subsection{Number of boundary nodal domains}
%%%%%%%%%%%%%%%%%%%%%%%%%%%%%%%%%%%%%%%%%
%%%%%%%%%%%%%%%%%%%%%%%%%%%%%%%%%%%%%%%%%
We recall that for $\delta, \epsilon$  fixed and $u$ being a Neumann eigenfunction for the chain domain, a boundary nodal domain, $D\in \mathcal{V}^\delta_1(\epsilon, u)$,  is one for which \eqref{e:boundary nd} holds. As introduced in \eqref{e:dom number}, we  write $\nu^{\delta}_1(\eps;u)$ for the number of boundary nodal domains of $u$.

In order to bound  $\nu^{\delta}_1(\eps;u)$, we use the argument leading to  \cite[Equation  (18)]{gittins-lena2020} and  \cite[Equation (14)] {lena-pleijel}. The first step in the proof is the following result.

\begin{lem} \label{lem:boundary-straighten}
Let $\Omega(\w)$ be a chain domain. There exist constants $c^*$, $C^*$  such that the following holds. Let $\delta>0$ satisfy \eqref{eqn:delta-partition}, $\varepsilon>0$, and $u$ be a Neumann eigenfunction of $\Omega(\w)$ with eigenvalue $\mu$. For each  $D \in \mathcal{V}^\delta_1(\epsilon, u)$
 there are a set ${V}_D$, with a Lipschitz boundary, and a function ${v}\in H_0^1({V}_D)$ with the following properties:
\begin{align} \label{eqn:boundary-area}
  \Area({V}_D) \leq C^*\Area(\{x\in D:\, \dist(x,\pa\Omega(\w))<\tfrac{3}{4}\tau^*\delta\}),
\end{align}
and 
\begin{align*}
   \int_{{V}_D}{v}^2 \geq c^*\eps, \qquad \int_{{V}_D}\left|\nabla{v}\right|^2 \leq C^*(\mu+(\tau^*\delta)^{-2}).
\end{align*}
\end{lem}
\begin{proof}
Lemma \ref{lem:boundary-straighten} follows from the straightening results of Lemma \ref{lem:boundary1} (with $\eta = \delta$) and Remark \ref{rem:neck-straighten}, using an identical argument to the proof of the estimates (13), (15), and (17) in Section 5 of \cite{gittins-lena2020}, and so we omit the details of the proof here.  In particular, this proof relies upon the Lipschitz properties of $D \cap \{u_1\neq0\}$ from Remark \ref{rmk:lip}.
\end{proof}

When bounding $\nu_1^\delta(\epsilon, u)$, we will use Lemma \ref{lem:boundary-straighten} along with Faber-Krahn's inequality to obtain that $\nu_1^\delta(\epsilon, u)$ is bounded by a multiple of $M_{\Omega(\w)}(\tfrac{3}{4}\tau^*\delta)$, where $M_{\Omega(\w)}(t)$ is as defined in \eqref{e:M}. To deal with this new upper bound, we will use the following lemma.

\begin{lem} \label{lem:M-bound}
Let $\Omega(\w)$ be a chain domain. There exists $C^*>0$ such that 
\begin{align*}
    M_{\Omega(\w)}(t) \leq C^*L(\w)t,
\end{align*}
for $t \leq \tfrac{3}{4}L(\w)\min\left\{\tau^*\delta^*,\tfrac{1}{\kappa^*}\right\}$.
\end{lem}

\begin{proof}
To bound $M_{\Omega(\w)}(t)$, we will break the function into three parts:
\begin{enumerate}

\item[1)] the contribution from the part of the necks $\{ \Q_{ij,k}(\w)\}$ that are a distance of at least $t$ from the boundary of the necks,
\item[2)] the contribution from the parts of $\{\pa \D_\ell\}$ with an appropriate neighborhood of each vertex excluded, 
\item[3)] the contribution from neighborhoods of the vertices of $\Omega(\w)$.
\end{enumerate}
We will denote the contribution to $M_{\Omega(\w)}(t)$ from each of these parts of $\pa\Omega(\w)$ by $M^{(j)}_{\Omega(\w)}(t)$, so that
\begin{align*}
    M_{\Omega(\w)}(t) = \sum_{j=1}^{3}M^{(j)}_{\Omega(\w)}(t).
\end{align*}

To handle 1), we use the diffeomorphisms $F_{ij,k}$ from Remark \ref{rem:neck-straighten} to straighten the upper and lower boundaries of each neck $\Q_{ij,k}(\w)$.  Then, there exists a constant $C^*$, depending only on the neck-width constant $w^*$ such that
\begin{align} \label{eqn:M-bound1}
M^{(1)}_{\Omega(\w)}(t) \leq C^*\sum_{1\leq i \leq j\leq M}\sum_{k=1}^{K}\text{Length}(\pa  \Q_{ij,k}(\w))t \leq C^*L(\w)t.
\end{align} 

For 2), we exclude an $\eta$ neighborhood of each vertex of $\D_\ell$, with $\eta$ given by $\eta = \tfrac{4}{3}(\tau^*)^{-1}t$. The upper bound on $t$ from the statement of the lemma then ensures that $\eta \leq L(\w)\min\left\{\delta^*,\tfrac{1}{\kappa^*\tau^*}\right\}$. Therefore, we can apply Lemma \ref{lem:boundary1} with this value of $\eta$ to obtain a diffeomorphism that straightens this part of the boundary of $\D_{\ell}$. Since this change of variables has a Jacobian bounded by $\tfrac{7}{4}$, we obtain
\begin{align} \label{eqn:M-bound2}
M^{(2)}_{\Omega(\w)}(t) \leq \tfrac{7}{4}\sum_{\ell=1}^{M}\text{Length}(\pa \D_\ell)t \leq \tfrac{7}{4}L(\w)t.
\end{align}
Finally, since, by the definition of the vertex control constant, $\Omega(\w)$ has at most $1/\delta^*$ vertices, the remaining contribution to \[\{x\in\Omega(\w):\dist(x,\pa\Omega(\w))<t\}\] is contained within $1/\delta^*$ discs of radius $2\eta$. Therefore,
\begin{align} \label{eqn:M-bound3}
M^{(3)}_{\Omega(\w)}(t) \leq 4\pi(\delta^*)^{-1}\eta^2 = \tfrac{64}{9}\pi (\delta^*)^{-1}(\tau^*)^{-2}t^2.
\end{align}
Combining \eqref{eqn:M-bound1}, \eqref{eqn:M-bound2}, and \eqref{eqn:M-bound3}, using the upper bound on $t$ gives the desired estimate on $M_{\Omega(\w)}(t)$. 
\end{proof}

Our main result in this section, which  bounds the number of boundary nodal domains  for an eigenfunction $u$, is the following.

\begin{prop}[Number of boundary nodal domains] \label{prop:boundary}
Let $\Omega(\w)$ be a chain domain. There exists a constant $C_1^*>0$ such that for all $0<\eps<\tfrac{1}{2}$ and $0<\beta<\tfrac{1}{2}$ the following holds. If $u$ is a Neumann eigenfunction of $\Omega(\w)$ with eigenvalue $\mu$ and $|\Omega(\w)|\mu \geq (C_1^*)^{1/\beta}$, then 
\begin{align*}
     \nu_1^\delta(\eps;u) \leq C_1^*\eps^{-1}\left(|\Omega(\w)|\mu\right)^{1-\beta}
\end{align*}
for $\delta=|\Omega(\w)|^{1/2-\beta}\mu^{-\beta}$.
\end{prop}

\begin{proof}
Let $D \in \mathcal{V}^\delta_1(\epsilon, u)$ and let $ V_{_{\!D}} \subset \Omega(\w)$ and $ v \in H_0^1( V_{_{\!D}})$  be as in  Lemma \ref{lem:boundary-straighten}. Then,  
\begin{equation}\label{e:V}
  \Area(V_{_{\!D}}) \leq C^*\Area(\{x\in D:\, \dist(x,\pa\Omega(\w))<\tfrac{3}{4}\tau^*\delta\}).
\end{equation}
On the other hand, writing $\lambda_1(\mathbb{D}_{V_{_{\!D}}})$ for the first Dirichlet eigenvalue of the disc of area equal to that of $V_{_{\!D}}$, the Faber-Krahn inequality yields
\begin{equation}\label{e:FK}
\lambda_1(\mathbb{D}_{V_{_{\!D}}}) \leq \lambda_1(V_{_{\!D}}) \leq  \frac{\int_{ V_{_{\!D}} }\left|\nabla {v}\right|^2 }{\int_{V_{_{\!D}} }{v^2}} \leq  C^*\eps^{-1}\left(\mu +(\tau^*\delta)^{-2}\right).
\end{equation}
Since $\lambda_1(\mathbb{D}_{V_{_{\!D}}})$ equals $\Area(V_{_{\!D}})^{-1}$ times a dimensional constant, we conclude from \eqref{e:V} that there is $C^*$ such that, for every $D \in \mathcal{V}^\delta_1(\epsilon, u)$,
$$
\eps\left(\mu +(\tau^*\delta)^{-2}\right)^{-1}\leq C^*\Area(\{x\in D:\, \dist(x,\pa\Omega(\w))<\tfrac{3}{4}\tau^*\delta\}).
$$
We then conclude
\begin{align} \label{eqn:boundary-area2}
\nu_1^\delta(\eps;u) \leq C^*\eps^{-1}\left(\mu + (\tau^*\delta)^{-2}\right)M_{\Omega(\w)}(\tfrac{3}{4}\tau^*\delta).
\end{align}

As explained in \eqref{eqn:mu-partition}, $\delta = |\Omega(\w)|^{1/2-\beta}\mu^{-\beta}$ satisfies \eqref{eqn:delta-partition} for $(|\Omega(\w)|\mu)^{1/\beta}$ sufficiently large, and so we can apply Lemma \ref{lem:M-bound} with $t = \tfrac{3}{4}\tau^*\delta$ to obtain
\begin{align*}
    M_{\Omega(\w)}
(\tfrac{3}{4}\tau^*\delta)\leq \tfrac{3}{4}C^*\tau^*L(\w)\delta &=  \tfrac{3}{4}C^*\tau^*L(\w)|\Omega(\w)|^{-1/2}|\Omega(\w)|(|\Omega(\w)|\mu)^{-\beta} \\
& \leq \tfrac{3}{4}C^*\tau^*(\rho^*)^{1/2}|\Omega(\w)|(|\Omega(\w)|\mu)^{-\beta}.
\end{align*} 
Here $\rho^*$ is the isoperimetric ratio constant.  In addition,
$$
    \mu + (\tau^*\delta)^{-2} = |\Omega(\w)|^{-1}\left(|\Omega(\w)|\mu+(\tau^*)^{-2}(|\Omega(\w)|\mu)^{2\beta}\right), 
$$ and so 
the result follows from \eqref{eqn:boundary-area2}.
\end{proof}

%%%%%%%%%%%%%%%%%%%%%%%%%%%%%%%%%%%%%%%%%
%%%%%%%%%%%%%%%%%%%%%%%%%%%%%%%%%%%%%%%%%
\section{Estimates on corner nodal domains} \label{sec:corner}
%%%%%%%%%%%%%%%%%%%%%%%%%%%%%%%%%%%%%%%%%
%%%%%%%%%%%%%%%%%%%%%%%%%%%%%%%%%%%%%%%%%
In this section, we will obtain an upper bound on the number of corner nodal domains.
Let $\Omega(\w)$ be a chain domain,  $\delta, \epsilon>0$, and  $u$ be a Neumann eigenfunction for $\Omega(\w)$. We recall that a corner nodal domain, $D\in \mathcal{V}^\delta_2(\epsilon, u)$,  is one for which \eqref{e:corner nd} holds. As introduced in \eqref{e:dom number}, we  write $\nu^{\delta}_2(\eps;u)$ for the number of corner nodal domains of $u$.

Given a vertex $v_0$ of $\Omega(\w)$ with interior angle $\theta_0$, we write 
$$
S_{v_0,\theta_0}(\delta)=\{(r,\theta) \in \Omega(\w): \; r<\delta, \,|\theta|<\theta_0/2\},
$$
where we use polar coordinates centered at $v_0$.  The first step in controlling $\nu^{\delta}_2(\eps;u)$ is the following analogue of Lemma \ref{lem:boundary-straighten}.

\begin{lem} \label{lem:corner-straighten}
Let $\Omega(\w)$ be a chain domain. There exist constants $c^*$, $C^*$  such that the following holds. Let $\delta>0$ satisfy  \eqref{eqn:delta-partition}, $\varepsilon>0$, and $u$ be a Neumann eigenfunction of $\Omega(\w)$ with eigenvalue $\mu$. Let  $D \in \mathcal{V}^\delta_2(\epsilon, u)$  such that $$\|u_2\|_{L^2(D)}^2 \geq \tfrac{1}{4}\eps\|u\|^2_{L^2(D)}.$$
 Then, there are a vertex $v_0$ of $\Omega(\w)$ with interior angle $\theta_0$, a set $V \subset S_{v_0,\theta_0}(\delta)$ with a Lipschitz boundary and a function $v\in H^1(V)$, with $v\equiv 0$ on $\pa V \cap S_{v_0,\theta_0}(\delta)$, and  the following properties:
\begin{align} \label{eqn:corner0a}
    \Area(V) \leq C^*\Area(\{x\in D:\chi_2^{\delta}u\neq0,\, \dist(x,v_0)<\delta\}),
\end{align}
 and
\begin{align} \label{eqn:corner0b}
   c^*\eps \leq \int_{V}v^2 \leq C^*, \qquad \qquad \int_{V}\left|\nabla v\right|^2 \leq C^*\left(\mu + \delta^{-2}\right).
\end{align}
\end{lem}

\begin{proof} By the definition of the vertex control constant $\delta^*$, the number of vertices of $\Omega(\w)$ is bounded by $1/\delta^*$. Therefore, since $\|u_2\|_{L^2(D)}^2 \geq \tfrac{1}{4}\eps\|u\|^2_{L^2(D)}$, there exists a vertex $v_0$, opening angle $\theta_0$, such that
\begin{align*}
    \int_{D\cap\{\dist(x,v_0)<\delta\}} u_2^2  \geq \tfrac{1}{4}\eps\delta^* \int_{D}u^2.
\end{align*}
We now apply a transformation to $u_2$ to translate and straighten the sides of $\Omega(\w)$ meeting at $v_0$. This will give us a function $v$ defined on $S_{v_0,\theta_0}(\delta)$ which is smooth away from the origin, and vanishes on the portion of the boundary of its support in the interior of the sector (i.e., $v\equiv 0$ on $\pa V \cap S_{v_0,\theta_0}(\delta)$). The transformation is constructed as follows:

By the upper bound on $\delta$ from \eqref{eqn:delta-partition}, and using the definition of the vertex control constant $\delta^*$, we can rotate and translate so that, without loss of generality, the vertex $v_0$ is at the origin, and in a disc of radius $20\delta$ centered at the vertex, this part of the boundary of $\pa \Omega(\w)$ can be written as $y = f_{\pm}(x)$ for $\pm x>0$. Moreover, there exists a constant $\theta^*$, depending only on $\delta^*$, such that the opening angle $\theta_0$ of each vertex satisfies  $\theta^*<\theta_0<2\pi-\theta^*$.  Therefore, the rotation can be chosen to ensure that $\lim_{x\to0^{\pm}}f_{\pm}(x)/x = a_{\pm}$, with
\begin{align*}
    c^* \leq |a_{\pm}| \leq C^*, \quad c^* \leq |f_{\pm}(x)/x| \leq C^*, \quad |f'_{\pm}(x)| \leq C^*,
\end{align*}
for constants $c^*$ and $C^*$ only depending on $\delta^*$.

Next, there exists a subset $S \subset S_{v_0,\theta_0}(\delta)$ such that the function
\begin{align*}
    F:  S  \to \Omega(\w), \qquad 
    (s,t)  \mapsto \left(s,\frac{f_{\pm}(s)}{a_{\pm}s}t\right),
\end{align*}
is a diffeomorphism onto its image, with the image containing the part of $\Omega(\w)$ in the disc of radius $\delta$ centered at $v_0$.

By Lemmas \ref{lem:partition} and \ref{lem:Green}, the function $u_2$ satisfies
\begin{align*}
    \int_{D}|\nabla u_2|^2 = \int_{D}\left|\nabla\left(\chi_2^\delta u\right)\right|^2 \leq C^*(\mu + \delta^{-2}).
\end{align*}
Therefore, we set $v = u_2\circ F$, and $V = F^{-1}(D\cap\{\dist(x,v_0)<\delta,u_2\neq0\})$, which by Remark \ref{rmk:lip} we may take to be Lipschitz. To complete the proof of the lemma, it is sufficient to show that the Jacobian of this change of variables is bounded from above and below, and that the off-diagonal entries are bounded. This Jacobian is given by the determinant of 
\begin{align*}
     \begin{pmatrix}
    1 & t\displaystyle{\frac{sf'_{\pm}(s)-f_{\pm}(s)}{a_{\pm} s^2}}\\
   0 & \displaystyle{\frac{f_{\pm}(s)}{a_{\pm}s}}.
    \end{pmatrix},
\end{align*}
and so by the properties of $f_{\pm}(s)/s$, we just need to bound the off-diagonal term. For some $|\zeta|$, $|\xi|\in(0,s)$ we have $f_{\pm}(s) = a_{\pm}s + \tfrac{1}{2}s^2f_{\pm}''(\zeta)$, and $f_{\pm}'(s) = a_{\pm} + sf_{\pm}''(\xi)$.  Therefore, it is sufficient to show that
\begin{align} \label{eqn:corner-straighten}
    |tf''_{\pm}(\zeta)| + |tf''_{\pm}(\xi)| \leq C^*
\end{align}
for a constant $C^*$. Since $f_{\pm}'(x)$ is bounded, and the curvature of the two sides of $\Omega(\w)$ is bounded from above by $\kappa^*/L(\w)$, we obtain
\begin{align*}
     |f''_{\pm}(x)| \leq C^*\kappa^*/L(\w).
\end{align*}
Combining this with the upper bound on $\delta$ from \eqref{eqn:delta-partition}, and since $|t|\leq \delta$, it implies \eqref{eqn:corner-straighten} and completes the proof of the lemma.
\end{proof}

We are now ready to state and prove the main bound on the number of corner nodal domains.

\begin{prop}[Number of corner nodal domains] \label{prop:corner}
Let $\Omega(\w)$ be a chain domain. There exists a constant $C_2^*>0$ such that for all $0<\eps<\tfrac{1}{2}$ and $0<\beta<\tfrac{1}{2}$ the following holds. If $u$ is a Neumann eigenfunction of $\Omega(\w)$ with eigenvalue $\mu$ and $|\Omega(\w)|\mu \geq (C_2^*)^{1/\beta}$, then 
\begin{align*}
    \nu_2^\delta(\eps, u) \leq C_2^*\eps^{-4}\left(|\Omega(\w)|\mu\right)^{3-6\beta},
\end{align*}
for $\delta=|\Omega(\w)|^{1/2-\beta}\mu^{-\beta}$.
\end{prop}

\begin{proof}
Throughout, we assume that $(|\Omega(\w)|\mu)^{\beta}$ is sufficiently large so that $\delta$ satisfies \eqref{eqn:delta-partition}, and in particular Lemmas \ref{lem:boundary-straighten} and \ref{lem:corner-straighten} apply.  The proof of the bound is divided into two cases, depending on whether we are counting nodal domains $D \in \mathcal{V}^\delta_2(\epsilon, u)$ for which either
\begin{equation*}
\text{(A)}\;\;  \|u_2\|_{L^2(D)}^2 \geq \tfrac{1}{4}\eps\|u\|^2_{L^2(D)} \qquad \text{or}\qquad \text{(B)}\;\;   \|u_4\|_{L^2(D\cap\cup_\ell \D_\ell)}^2 \geq \tfrac{1}{8}\eps\|u\|^2_{L^2(D)}.
\end{equation*}
In Case 1 below, we prove that 
\begin{equation}\label{e:AA}
   \#\{D \in \mathcal V_2^\delta(\eps, u): (A) \,\text{holds}\} \leq C_2^*\eps^{-4}\left(|\Omega(\w)|\mu\right)^{3-6\beta},
\end{equation}
and we explain how to do the same with $(A)$ replaced by $(B)$ in Case 2 (see \eqref{e:BB}). The bound on $\nu_2^\delta(\eps, u)$ then follows from these two estimates.

\noindent\emph{\underline{Case 1: $D \in \mathcal{V}^\delta_2(\epsilon, u)$ is such that (A) holds.}} 

First, we claim that there exist constants $c^*$, $C^*_2$ such that, if $|\Omega(\w)|\mu \geq (C_2^*)^{1/\beta}$, then
\begin{align}\label{claim:corner}
    \Area\left(D\cap B_{\delta}(v_0)\right)\geq c^* \eps^{4} |\Omega(\w)|(|\Omega(\w)|\mu)^{4 \beta - 3}.
\end{align}
Here $B_{\delta}(v_0)$ is the disc of radius $\delta$ centered at the vertex $v_0$.
Since $B_{\delta}(v_0)$ has area $\pi \delta^{-2} = \pi  |\Omega(\w)|^{1-2\beta}\mu^{-2\beta}$, the bound in \eqref{claim:corner} yields \eqref{e:AA}
 as claimed.

We next proceed to prove
 the claim in \eqref{claim:corner}.
 Let $v, V$ be as in Lemma \ref{lem:corner-straighten}.
By the lower bound in \eqref{eqn:corner0b}, there exists $x^*$, with $|x^*|<\delta$ such that 
\begin{align} \label{eqn:corner2}
    \int_{V\cap\{x=x^*\}} v(x^*,y)^2\,dy \geq \tfrac{1}{2}c^*\eps \delta^{-1}.
\end{align}
We now extend $v$ identically by zero outside of $V$. Note that this does not give a function in $H^1(\mathbb{R}^2)$, because $v$ does not satisfy Dirichlet boundary conditions on the part of $\pa V$ coinciding with the sector $\theta = \pm\tfrac{1}{2}\theta_0$. However, denoting this extension of $v$ by $w$, for each $x^*$ such that $(x^*,x_2)$ is in $V$, we do have 
\begin{align} \label{eqn:extension}
    w(\cdot,y)\in H^1([x^*,x^*+\delta]),
\end{align}
with
\begin{align} \label{eqn:extension1}
    \int_{V\cap\{x=x^*\}}\left(\int_{x^*}^{x^*+\delta}|\nabla w|^2\,dx\right)\,dy \leq C^*(\mu+\delta^{-2}).
\end{align}
 This is because for all such $y$, the interval $(x^*,x^*+ \delta)$ is contained in interior of the sector, where $v$ is a smooth function. Therefore, since $v$ vanishes on the portion of $\pa D$ in the interior of $\Omega$, its extension by $0$ satisfies \eqref{eqn:extension}, $w(\cdot,y)$ is continuous on $(x^*,x^*+ \delta)$, and \eqref{eqn:extension1} follows from the second estimate in \eqref{eqn:corner0b}. In particular, $\pa_tw(t,y)$ is integrable for $t\in(x^*,x^*+ \delta)$, and this means that by the fundamental theorem of calculus
\begin{align*}
    w(x,y) = w(x^*,y) + \int_{x^*}^{x}\pa_tw(t,y)\,dt
\end{align*}
for $x\in (x^*,x^*+\delta)$. We can write
\begin{align*}
    \left|\int_{x^*}^{x}\pa_tw(t,y)\,dt\right| \leq |x-x^*|^{1/2}\left(\int_{x^*}^{x}|\pa_tw(t,y)|^2\,dt\right)^{1/2},
\end{align*}
and so 
\begin{align*}
    w(x,y)^2 & = w(x^*,y)^2 + 2w(x^*,y)\int_{x^*}^{x}\pa_tw(t,y)\,dt + \left(\int_{x^*}^{x}\pa_tw(t,y)\,dt\right)^2 \\
    & \geq \tfrac{1}{2}w(x^*,y)^2 - \left(\int_{x^*}^{x}\pa_tw(t,y)\,dt\right)^2 \\
    & \geq  \tfrac{1}{2}w(x^*,y)^2 - |x-x^*|\int_{x^*}^{x}|\pa_tw(t,y)|^2\,dt.
\end{align*}
Integrating in $y$ and using \eqref{eqn:corner2} this implies that
\begin{align} \label{eqn:corner3}
    \int_{V\cap\{x=x^*\}}w(x,y)^2\,dy \geq \tfrac{1}{4}\eps c^*\delta^{-1} - |x-x^*|\int_{V\cap\{x=x^*\}}\left(\int_{x^*}^{x}|\pa_tw(t,y)|^2\,dt\right)\,dy.
\end{align}
Using the estimate from \eqref{eqn:extension1} in \eqref{eqn:corner3} then gives
\begin{align*}
     \int_{V\cap\{x=x^*\}}w(x,y)^2\,dy \geq \tfrac{1}{4}\eps c^*\delta^{-1}- C^*|x-x^*|(\mu + \delta^{-2}).
\end{align*}
So, 
\begin{align} \label{eqn:corner4}
\int_{V\cap\{x=x^*\}}v(x,y)^2\,dy = \int_{V\cap\{x=x^*\}}w(x,y)^2\,dy \geq \tfrac{1}{8}c^*\eps \delta^{-1}
\end{align}
for all $x\geq x^*$ with $|x-x^*|\leq\tfrac{1}{8}c^*\eps \delta^{-1} (C^*)^{-1}(\mu+ \delta^{-2})^{-1}$.  For simplicity, we define the new parameter
\[
\tilde{\delta} = c^*(C^*)^{-1}\delta^{-1}  (\mu + \delta^{-2})^{-1}.
\]
Note that we have set $\delta = |\Omega(\w)|^{1/2-\beta}\mu^{-\beta}$, and so $\tilde{\delta} \sim \delta^{-1}\mu^{-1} = |\Omega(\w)|^{1/2}(|\Omega(\w)|\mu)^{\beta-1}$ for large $|\Omega(\w)|\mu$ and $0<\beta<\tfrac{1}{2}$. In particular, $\tilde{\delta}<\delta$ for large $|\Omega(\w)|\mu$.

Now let $\chi = \chi(x,y)$ be a smooth cut-off function, centred at $(x^* + \tfrac{1}{16}\eps \tilde{\delta},0)$, and with the following properties: It equals $1$ in a disc of radius $c^*_0 \epsilon \tilde{\delta}$ centred at this point, vanishes outside the disc of radius $2c^*_0 \epsilon \tilde{\delta}$, with first derivatives bounded by $C_0^* \epsilon^{-1} \tilde{\delta}^{-1}$. Here $c^*_0>0$ and $C_0^*$ are chosen, depending only on the interior angle $\theta_0$ at the vertex (which is bounded from below by a constant which only depends on $\delta^*$), so that the support of $\chi$ is contained within the interior of the sector. Using \eqref{eqn:corner4}, it can also be chosen so that
\begin{align*}
\int_{V}\chi^2w^2 \geq c^* \eps^2 \tilde{\delta} \delta^{-1}% = \epsilon^{2} \mu^{2 \beta -1}
\end{align*}
for a constant $c^*>0$. Here, and from now on, the constants $c^*$, $C^*$ may change from line-to-line (but will depend only on the five geometric constants of $\Omega(\w)$).  Let $W_{\chi}$ be the support of $\chi$. Then, the definition of $\chi$ and \eqref{eqn:corner0b} also ensures that
\begin{align*}
    \int_{V\cap W_{\chi}}\left|\nabla\left(\chi w\right)\right|^2 \leq C^*(\mu +\delta^{-2}+ \tilde{\delta}^{-2} \epsilon^{-2}).
\end{align*}
Since $W_{\chi}$ is contained in the interior of the sector, $\chi w$ vanishes on the boundary of $V\cap W_{\chi}$. This implies that the first Dirichlet eigenvalue of $V\cap W_{\chi}$ is bounded by a multiple of 
\begin{align} \label{eqn:corner-end}
\frac{\int_{V\cap W_{\chi}}\left|\nabla\left(\chi w\right)\right|^2}{\int_{V\cap W_{\chi}}\chi^2w^2 } \leq  C^*\eps^{-2}\tilde{\delta}^{-1}\delta(\mu +\delta^{-2}+ \tilde{\delta}^{-2} \epsilon^{-2}).
\end{align}
Since $\delta = |\Omega(\w)|^{1/2-\beta}\mu^{-\beta}$, for $0<\beta<\tfrac{1}{2}$, and for $(|\Omega(\w)|\mu)^{\beta}\geq C^*$ sufficiently large, the right hand side of \eqref{eqn:corner-end} can be bounded from above by
\begin{align} \label{eqn:corner-end1}
   C^*\eps^{-4}\delta^4\mu^3 = C^*\eps^{-4}|\Omega(\w)|^{-1}(|\Omega(\w)|\mu)^{3-4\beta}.
\end{align}
Therefore, using the Faber-Krahn inequality as in \eqref{e:FK},
\begin{align*}
    \text{Area}(V\cap W_{\chi}) \geq  c^*\eps^{4}|\Omega(\w)|(|\Omega(\w)|\mu)^{4\beta-3},
\end{align*}
completing the proof of the claim in \eqref{claim:corner}.

\noindent \underline{\emph{Case 2: $D \in \mathcal{V}^\delta_2(\epsilon, u)$ is such that $(B)$ holds.}}

In this case, some of the mass of $u$ is contained in the intersection of the support of $\chi_4$ with a particular domain $\D_\ell$ (near where $\D_\ell$ and a neck $ \Q_{ij,k}(\w)$ are joined).  Then, we straighten the part of $\pa \D_\ell$ using Lemma \ref{lem:boundary-straighten} 
to again obtain, after a rotation, a Lipschitz set $V$ contained in the half-plane $\{x>0\}$, and a function $v\in H^1(V)$ with 
\begin{align*} 
    \text{Area}(V) \leq C^*\text{Area}(\{x\in D\cap \D_\ell:\chi_4^{\delta}u\neq0\}),
\end{align*}
and
\begin{align*} 
   c^*\eps \leq \int_{V}v^2 \leq C^*, \qquad \int_{V}\left|\nabla v\right|^2 \leq C^*\left(\mu + \delta^{-2}\right).
\end{align*}
Moreover, $v$ can be taken to vanish on $\pa V\cap\{x>0\}$. This then allows us to replicate the proof of Case 1) and show that 
\begin{equation}\label{e:BB}
   \#\{D \in \mathcal V_2^\delta(\eps, u): (B) \,\text{holds}\} \leq C_2^*\eps^{-4}\left(|\Omega(\w)|\mu\right)^{3-6\beta},
\end{equation}
as needed to finish the proof of the proposition.

\end{proof}

%%%%%%%%%%%%%%%%%%%%%%%%%%%%%%%%%%%%%%%%%
%%%%%%%%%%%%%%%%%%%%%%%%%%%%%%%%%%%%%%%%%
\section{Estimates on neck nodal domains} \label{sec:neck}
%%%%%%%%%%%%%%%%%%%%%%%%%%%%%%%%%%%%%%%%%
%%%%%%%%%%%%%%%%%%%%%%%%%%%%%%%%%%%%%%%%%
In this section, we will obtain an upper bound on the number of neck nodal domains.
Let $\Omega(\w)$ be a chain domain,  $\delta, \epsilon>0$, and  $u$ be a Neumann eigenfunction for $\Omega(\w)$. We recall that a neck nodal domain, $D\in \mathcal{V}^\delta_3(\epsilon, u)$,  is one for which \eqref{e:neck nd} holds. As introduced in \eqref{e:dom number}, we  write $\nu^{\delta}_3(\eps;u)$ for the number of neck nodal domains of $u$.

This section is divided into three parts. Given $D\in\mathcal{V}^\delta_3(\epsilon, u)$, in Section \ref{s:straight L} we explain how to find a neck $\mathcal N$ for which there is a lower bound on the area of $\mathcal{N}\cap D$.  This lower bound will be given in the form of $\Area(V)$ where $V$ is a subset of a flat cylinder or a strip into which the neck has been straightened. Then, in Section \ref{s: cyl} we explain how to find a lower bound on $\Area(V)$, when $V$ is a subset of a cylinder, in terms of its first Dirichlet eigenvalue. Finally, in Section \ref{s:neck res} we state and prove the bound on $\nu^{\delta}_3(\eps;u)$.

%%%%%%%%%%%%%%%%%%%%%%%%%%%%%%%%%%%%%%%%%
\subsection{Straightening lemmas}\label{s:straight L}
%%%%%%%%%%%%%%%%%%%%%%%%%%%%%%%%%%%%%%%%%

Given a nodal domain  $D \in \mathcal{V}^\delta_3(\epsilon, u)$, we know that one of the following two inequalities hold:
\begin{equation}\label{e: tildeAB}
{(\tilde{A})}\;\;  \|u_3\|_{L^2(D)}^2 \geq \tfrac{1}{4}\eps\|u\|^2_{L^2(D)} \qquad \text{or}\qquad {(\tilde{B})}\;\;   \|u_4\|_{L^2(D\cap \cup  \Q_{ij,k}(\w))}^2 \geq \tfrac{1}{8}\eps\|u\|_{L^2(D)}^2.
\end{equation}
In the following lemmas, we explain how to find a neck $\mathcal N$ for which there is a lower bound on the area of $\mathcal{N}\cap D$ in both the  $(\tilde{A})$ and $(\tilde{B})$ cases.

\begin{lem} \label{lem:neck-straighten}
Let $\Omega(\w)$ be a chain domain. There exist constants $c^*$, $C^*$  such that the following holds. Let $\delta>0$ satisfy  \eqref{eqn:delta-partition}, $\varepsilon>0$, and $u$ be a Neumann eigenfunction of $\Omega(\w)$ with eigenvalue $\mu$. Let  $D \in \mathcal{V}^\delta_3(\epsilon, u)$  such that $(\tilde{A})$ holds.
Then, there exist a neck $ \Q_{ij,k}(\w)$, a set $V$ with a Lipschitz boundary that is a subset of a flat cylinder of circumference $4w_{ij,k}$, and a function $v\in H_0^1(V)$ with the following properties:
\begin{align} \label{eqn:neck1}
    \Area(V) \leq C^*\Area(\{x\in D\cap  \Q_{ij,k}(\w):\chi_3^\delta u\neq0\})
\end{align}
and 
\begin{align} \label{eqn:neck2}
    c^*\eps \leq \int_{V}v^2\leq C^*, \qquad \int_{V}|\nabla v|^2 \leq C^*(\mu+\delta^{-2}).
\end{align}
\end{lem}

\begin{proof}
As $1/\delta^*$ is an upper bound on the number of vertices of $\Omega(\w)$, and each neck contributes four vertices, $\Omega(\w)$ has at most $1/(4\delta^*)$ necks. Since $\|u_3\|_{L^2(D)}^2 \geq \tfrac{1}{4}\eps\|u\|^2_{L^2(D)}$, there therefore exists a neck $ \Q_{ij,k}(\w)$ such that
\begin{align*}
    \int_{D\cap  \Q_{ij,k}(\w)}u_3^2  \geq \eps\delta^*\int_{D}u^2.
\end{align*}
 We note that the support of the neck cut-off function $\chi_3^{\delta}$ only intersects a neck $ \Q_{ij,k}(\w)$  when $w_{ij,k}\leq 4\delta$ (see Definition \ref{defn:partition}). Recall that $w_{ij,k}$ is the minimum width of the neck $\Q_{ij,k}(\w)$.

As we commented in Remark \ref{rem:neck-straighten}, by the definition of the neck $\Q_{ij,k}(\w)$, there exists a diffeomorphism $F_{ij,k}$, 
\begin{align*}
    F_{ij,k}:[0,L_{ij,k}]&\times (-w_{ij,k},w_{ij,k}) \to \Q_{ij,k}(\w) \\
       F_{ij,k}(s,t) & = \G_{ij,k}\left(s,t_1 + (t+w_{ij,k}) |I_{ij,k}|/(2w_{ij,k})\right).
\end{align*} 
The Jacobian of $F_{ij,k}$ is bounded from above and below by a constant depending only on the neck-width constant $w^*$. Then, we define the function $\tilde u := u_3 \circ F_{ij,k}$ and set $\tilde{V}$ given by
\begin{align*} 
F^{-1}_{ij,k}(D\cap\{u_3\neq0\}) \subset [0,L_{ij,k}]\times(-w_{ij,k},w_{ij,k}).
\end{align*}
Note that $\tilde{V}$ has Lipschitz boundary and meets the lines $t=\pm w_{ij,k}$ at non-zero angles, by Remark \ref{rmk:lip}. We now reflect $\tilde{u}$ and $\tilde{V}$ across the line $t=w_{ij,k}$, and glue across the line $t=-w_{ij,k}$. This gives a subset $V$ of a cylinder of circumference $4w_{ij,k}$, with Lipschitz boundary, and a function $v\in H^1_0(V)$, with the required properties.
\end{proof}

\begin{lem} \label{lem:neck-straighten2}
Let $\Omega(\w)$ be a chain domain. There exist constants $c^*$, $C^*$  such that the following holds. Let $\delta>0$ satisfy  \eqref{eqn:delta-partition}, $\varepsilon>0$, and $u$ be a Neumann eigenfunction of $\Omega(\w)$ with eigenvalue $\mu$. Let  $D \in \mathcal{V}^\delta_3(\epsilon, u)$  such that $(\tilde{B})$ holds.

Then, there exist a neck $ \Q_{ij,k}(\w)$, a set $V$ with a Lipschitz boundary that is a subset of a strip of width $2w_{ij,k}$ and length $C^*\delta$,
 and a function $v\in H^1(V)$ with the following properties: 
\begin{align*}
    \Area(V) \leq C^*\Area(\{x\in D\cap  \Q_{ij,k}(\w)\,:\, \chi_4^{\delta}u \neq0\})
\end{align*}
and 
\begin{align} \label{eqn:neck-straighten2a}
    c^*\eps \leq \int_{V}v^2\leq C^*, \quad \int_V|\nabla v|^2 \leq C^*(\mu+\delta^{-2}).
\end{align}
\end{lem}

\begin{proof}
This lemma follows using the same idea as for the proof of Lemma \ref{lem:neck-straighten}: There exists a neck $ \Q_{ij,k}(\w)$ such that the lower bound
\begin{align*}
    \int_{D\cap  \Q_{ij,k}(\w)}u_4^2\,dx \geq \tfrac{1}{2}\eps\delta^*\int_{D}u^2\,dx
\end{align*}
holds. We can use the same diffeomorphism $F_{ij,k}$ to transform this neck to an exact strip of width $2w_{ij,k}$, with a bounded Jacobian. Since the cut-off function $\chi_4^{\delta}$ is supported in $\delta$-neighborhood of the ends of the neck $\G_{ij,k}(\{0\}\times I_{ij,k})$ and $\G_{ij,k}(\{L_{ij,k}\}\times I_{ij,k})$, the function $v=  u_4\circ F_{ij,k}$ and set $V = F_{ij,k}^{-1}(D\cap\Q_{ij,k}(\w)\cap\{u_4\neq0\})$ satisfy the estimates given in the statement of the lemma. In particular, as we commented in Remark \ref{rmk:lip}, this ensures that the set $V$ has Lipschitz boundary.
\end{proof}

%%%%%%%%%%%%%%%%%%%%%%%%%%%%%%%%%%%%%%%%%
\subsection{Controlling the area of a subset of a cylinder}\label{s: cyl}
%%%%%%%%%%%%%%%%%%%%%%%%%%%%%%%%%%%%%%%%%

Given $D\in \mathcal{V}^\delta_3(\epsilon, u)$ satisfying $(\tilde A)$, Lemma \ref{lem:neck-straighten} gives the existence of a neck $\mathscr{N}_{ij,k}(\w)$  so that the area of $\mathscr{N}_{ij,k}(\w) \cap D$ can be bounded below by the area  $\Area(V)$ where $V$ is a subset of a flat cylinder. In this section we explain how to obtain a lower bound on $\Area(V)$ in terms of the first Dirichlet eigenvalue for $V$.

\begin{lem} \label{lem:cylinder}
Let $\mathscr{C}_{_{\!P}}$ be the flat infinite cylinder with circumference $P>0$, and let $V$ be a Lipschitz set on $\mathscr{C}_{_{\!P}}$ with first Dirichlet eigenvalue equal to $\lambda$. Then,
\begin{align*}
    \Area(V) \geq \min\{\pi\lambda_1(\mathbb{D})\lambda^{-1}, P\lambda_1(\mathbb{D})^{1/2}\lambda^{-1/2} \}.
\end{align*}
\end{lem}

\begin{proof}
To prove the lemma, we use the classical proof of the Faber-Krahn inequality in the case of constant sectional curvature $\kappa=0$ (see \cite{chavel84}). However, instead of using the isoperimetric inequality in $\mathbb{R}^2$, we use the following inequality on the flat cylinder \cite[Theorem 6]{howards1999}: 
Let $W$ be a region on the cylinder $\mathscr{C}_{_{\!P}}$, enclosing area $A$ and of perimeter $L$. Then,
\begin{equation}\label{e:iso}
1) \;\;\; L\geq 2\sqrt{\pi A}\; \;\;\text{if}\;\;\; A\leq\tfrac{1}{\pi}P^2,\qquad \qquad2) \;\;\; L \geq 2P \;\;\;\text{if}\;\;\; A\geq\tfrac{1}{\pi}P^2,
\end{equation}
with equality only when $W$ is an isometric embedding of a round disc on the cylinder in the first case, and the region between two cross-sections of $\mathscr{C}_{_{\!P}}$ in the second case.

Let $v\in H_0^1(V)$ be a non-negative first Dirichlet eigenfunction of $V$, with eigenvalue $\lambda$. We then build a comparison function with circular level sets as follows: Let $\mathbb{D}(t)\subset \mathbb{R}^2$ be the disc of radius $r(t)$, such that  $\pi r(t)^2=|\mathbb{D}(t)|=\text{Area}(\{v>t\})$. Note that $r:[0, t_0] \to [0, r(0)]$,  with $t_0=\max v$ and $r(0)=\text{Area}(V)$, is continuous and strictly decreasing.

Next, let $\Psi=r^{-1}$, and define $F:\mathbb{D}(0) \to \mathbb R $ by  $F(p)=\Psi(|p|)$. Note that by the co-area formula
\begin{align}\label{e:coArea}
\int_{\{v=t\}} \tfrac{1}{|\nabla v| } d\sigma_t=-\tfrac{d}{ds}\big(\text{Area}(\{v>s\})\big)\big|_{s=t}  =-2\pi  r(t) r'(t),
\end{align}
where $d\sigma_t$ is the measure on the level set $\{v=t\}\subset V$. Therefore, since $t= \Psi(r(t))$,
\begin{align}\label{e:norm}
\int_D v^2 \,\text{dv}_g
&= \int_{0}^{t_0} t^2 \int_{\{v=t\}} \tfrac{1}{|\nabla v| } d\sigma_t dt  
 =  - 2 \pi  \int_{0}^{t_0} (\Psi(r(t)))^2 r(t) r'(t)dt
=   \int_{ \mathbb{D}(0) }|F|^2 \text{dv}.
\end{align}
Here we have used $\text{dv}_g$ and $\text{dv}$ to denote the area measure on $V$ and $\mathbb{R}^2$ respectively. Next, notice that by Cauchy-Schwartz $\left( \int_{\{v=t\}} \tfrac{1}{|\nabla v| } d\sigma_t \right)  \left( \int_{\{v=t\}} {|\nabla v| } d\sigma_t\right)  \geq \text{Length}(\{v=t\})^2$. Thus,  by \eqref{e:coArea} we have 
\begin{align*}
\int_{\{v=t\}} {|\nabla v| } d\sigma_t \geq  - \frac{\text{Length}(\{v=t\})^2}{2\pi  r(t) r'(t) }. 
\end{align*}
Using again the co-area formula 
\begin{align}\label{e:preIso}
\int_D|\nabla v|^2 \,\text{dv}_g
&= \int_{0}^{t_0} \int_{\{v=t\}} {|\nabla v| } d\sigma_t dt   
\geq - \int_{0}^{t_0}   \frac{\text{Length}(\{v=t\})^2}{2\pi  r(t) r'(t) } dt.
\end{align}
We now split into two cases depending on the relative size of Area$(\{v>t\})$ and the circumference $P$ of the cylinder $\mathscr{C}_{_{\!P}}$. 

\textbf{Case 1.} Suppose that the area of $V$ satisfies
\begin{align*}
    \text{Area}(V) \leq \tfrac{1}{\pi}P^2.
\end{align*}
Then, since $\text{Area}(\{v>t\})\leq \text{Area}(V)$,  case 1) in \eqref{e:iso} implies that 
$$
\text{Length}(\{v=t\}) \geq 2\sqrt{\pi A(t)} =  2\pi r(t).
$$
 Thus, 
using that $\Psi'(r(t))r'(t)=1$, the bound in \eqref{e:preIso} yields
\begin{align}\label{e:grad1}
\int_D|\nabla v|^2 \,\text{dv}_g
\geq - \int_{0}^{t_0}   \frac{2\pi r(t)}{r'(t) } dt=  - \int_{0}^{t_0}  2\pi r(t)(\Psi'(r(t)))^2 r'(t)  dt= \int_{ \mathbb{D}(0)}|\nabla F|^2 \,\text{dv}.
\end{align}
From \eqref{e:norm} and \eqref{e:grad1} we have  
\begin{align}\label{eqn:FK-1}
    \lambda \geq  \lambda_1(\mathbb{D}(0))= \frac{\pi }{\text{Area}(V)} \lambda_1(\mathbb{D}),
\end{align}
which can be rearranged to give $\text{Area}(V) \geq \pi\lambda_1(\mathbb{D})\lambda^{-1}$.

\textbf{Case 2.} Suppose now that the area of $V$ instead satisfies 
\begin{align*}
    \text{Area}(V) \geq \tfrac{1}{\pi}P^2.
\end{align*}
Set $A(t) = \text{Area}(\{v>t\})$. Then, the isoperimetric inequality in \eqref{e:iso} implies that
\begin{align*}
   \text{Length}(\{v=t\}) \geq \min\left\{2P,2\sqrt{\pi A(t)}\right\} = \min\left\{\frac{P}{\pi r(t)},1\right\} 2\pi r(t)\geq \frac{P}{\sqrt{\pi \text{Area}(V)}} 2\pi r(t).
\end{align*}
To get the last inequality we used that  $\frac{P}{\pi r(t)} \geq \frac{P}{\pi r(0)}=\frac{P}{\sqrt{\pi \text{Area}(V)}}$, together with  $\frac{P}{\sqrt{\pi \text{Area}(V)}} \leq 1$.
Thus, 
using that $\Psi'(r(t))r'(t)=1$, the bound in \eqref{e:preIso} yields
\begin{align}\nonumber
\int_D|\nabla v|^2 \,\text{dv}_g
\geq - \frac{P^2}{\pi \text{Area}(V)} \int_{0}^{t_0}   \frac{2\pi r(t)}{r'(t) } dt& =  -   \frac{P^2}{\pi \text{Area}(V)}\int_{0}^{t_0}  2\pi r(t)(\Psi'(r(t)))^2 r'(t)  dt\\\label{e:grad2}
& =   \frac{P^2}{\pi \text{Area}(V)} \int_{ \mathbb{D}(0)}|\nabla F|^2 \,\text{dv}.
\end{align}
From \eqref{e:norm} and \eqref{e:grad2} we have  
\begin{align} \label{eqn:FK-2}
    \lambda \geq   \tfrac{P^2}{\pi \text{Area}(V)} \lambda_1(\mathbb{D}(0))=  \tfrac{P^2}{\text{Area}(V)^2} \lambda_1(\mathbb{D}).
\end{align}
This can be rearranged to give $\text{Area}(V) \geq P\lambda_1(\mathbb{D})^{1/2}\lambda^{-1/2}$. Since either \eqref{eqn:FK-1} or \eqref{eqn:FK-2} must hold, this completes the proof of the proposition.
\end{proof}

\begin{remark}
We do not expect the second lower bound in Lemma \ref{lem:cylinder} to be sharp: The first lower bound of $\Area(V)\geq\pi\lambda_1(\mathbb{D})\lambda^{-1}$ is the same as the lower bound from the Faber-Krahn theorem for the disc. For the section $S_A$ on the cylinder $\mathscr{C}_{_{\!P}}$ of area $A$, the first Dirichlet eigenfunction is $\sin\left( \tfrac{P \pi x}{A}\right)$, with eigenvalue $\lambda = P^2\pi^2 A^{-2}$. Therefore, in this case, we have the equality $\Area(S_A) = P\pi \lambda^{-1/2}$.

Motivated by this, we conjecture that
\begin{align*}
     \Area(V) \geq \min\{\pi\lambda_1(\mathbb{D})\lambda^{-1}, P\pi\lambda^{-1/2} \},
\end{align*}
with the minimizer given by a disc on $\mathscr{C}_{_{\!P}}$ if $\lambda \geq \lambda_1(\mathbb{D})^2P^{-2}$, and a section of $\mathscr{C}_{_{\!P}}$ if $\lambda \leq \lambda_1(\mathbb{D})^2 P^{-2}$. The second lower bound on $\Area(V)$ from the proposition is off from this conjectured sharpest lower bound by a factor of
\begin{align*}
    \pi \lambda_1(\mathbb{D})^{-1/2} = \pi j_{0,1}^{-1} \approx 1.306.
\end{align*}
As this factor is independent of $\lambda$ and $P$, Lemma \ref{lem:cylinder} is sufficient for our nodal domain count estimate.
\end{remark}

%%%%%%%%%%%%%%%%%%%%%%%%%%%%%%%%%%%%%%%%%
\subsection{Number of neck nodal domains}\label{s:neck res}
%%%%%%%%%%%%%%%%%%%%%%%%%%%%%%%%%%%%%%%%%
We are now ready to state and prove our main result for this section.

\begin{prop}[Number of neck nodal domains] \label{prop:neck}
Let $\Omega(\w)$ be a chain domain. There exists a constant $C_3^*>0$ such that for all $0<\eps<\tfrac{1}{2}$ and $0<\beta<\tfrac{1}{2}$ the following holds.
If $u$ is a Neumann eigenfunction of $\Omega(\w)$ with eigenvalue $\mu$ and $|\Omega(\w)|\mu \geq (C_3^*)^{1/\beta}$, then 
\begin{align*}
     \nu_3^\delta(\eps;u) \leq C_3^*\left[ \eps^{-1}(|\Omega(\w)|\mu)^{1-\beta}+\eps^{-4}(|\Omega(\w)|\mu)^{3-6\beta}\right],
\end{align*}
for $\delta=|\Omega(\w)|^{1/2-\beta}\mu^{-\beta}$.
\end{prop}

\begin{proof}
Throughout, we assume that $(|\Omega(\w)|\mu)^{\beta}$ is sufficiently large so that $\delta$ satisfies \eqref{eqn:delta-partition}, and in particular Lemmas \ref{lem:neck-straighten} and \ref{lem:neck-straighten2} apply. The collection of neck nodal domains $\mathcal{V}^\delta_3(\epsilon, u)$ is split into those who satisfy either $(\tilde A)$ or $(\tilde B)$ in \eqref{e: tildeAB}. We proceed to prove that the upper bound we claim on $\nu_3^\delta(\eps;u)$ holds in each case.

\noindent \underline{\emph{Case 1: $D \in \mathcal{V}^\delta_3(\epsilon, u)$ is such that $(\tilde A)$ holds.}}

In this case, by Lemma \ref{lem:neck-straighten}, there exist a neck $\mathscr{N}(\w):= \Q_{ij,k}(\w)$, of minimum width $w:=w_{ij,k}$ and a Lipschitz set $V$ that is a subset of a flat cylinder of circumference $4w$ such that
$$
\Area(V) \leq C^*\Area(\{x\in D\cap \mathscr{N}(\w) :\chi_3^\delta u\neq0\}).
$$
Here, to simplify notation, we have dropped the $ij,k$ subscripts.  Therefore, Lemma \ref{lem:cylinder} yields
$$
\min\{\pi\lambda_1(\mathbb{D})\lambda^{-1}, P\lambda_1(\mathbb{D})^{1/2}\lambda^{-1/2} \}
\leq C^*\Area(\{x\in D\cap \mathscr{N}(\w) :\chi_3^\delta u\neq0\}),
$$
where $\lambda$ is the first Dirichlet eigenvalue for $V$.
By  \eqref{eqn:neck2}, we have
\begin{align*}
    \lambda \leq C^*\eps^{-1}(\mu+\delta^{-2})=:C^*\gamma,
\end{align*}
and so
\begin{align*}
    \Area(\{x\in D\cap \mathscr{N}(\w) :\chi_3^\delta u\neq0\}) \geq (C^*)^{-1}\min\{\pi\lambda_1(\mathbb{D})\gamma^{-1}, 4w\lambda_1(\mathbb{D})^{1/2}\gamma^{-1/2} \}
\end{align*}
for a constant $C^*$ that may increase from line-to-line. The area of the neck $\Q(\w)$ is bounded by a constant depending only on the neck-width constant $w^*$ multiplied by $wL(\w)$.  Therefore,
\begin{align} \label{eqn:neck4b}
    \#\{D \in \mathcal V_3^\delta(\eps, u): (\tilde A) \,\text{holds}\}\leq  C^*wL(\w)\max\{\gamma, w^{-1}\gamma^{1/2}\}.
\end{align}
For $\delta = |\Omega(\w)|^{1/2-\beta}\mu^{-\beta}$, we have
\begin{align*} %\label{eqn:neck4}
   \gamma= \epsilon^{-1} |\Omega(\w)|^{-1}\left(|\Omega(\w)|\mu + \left(|\Omega(\w)|\mu\right)^{2\beta} \right).
\end{align*}
By Definition \ref{defn:partition}, as the support of $\chi_3^\delta$ intersects $\Q(\w)$, we must have $w\leq 4\delta$. Therefore, since the isoperimetric ratio $L(\w)^2/|\Omega(\w)|$ is bounded from above by $\rho^*$, we have
\begin{align*}
wL(\w)\gamma  \leq 4\delta L(\w)\gamma & = 4L(\w)|\Omega(\w)|^{-1/2}\left(|\Omega(\w)|\mu\right)^{-\beta}\eps^{-1}\left(|\Omega(\w)|\mu + \left(|\Omega(\w)|\mu\right)^{2\beta} \right) \\
& \leq 4(\rho^*)^{1/2}\eps^{-1}\left(\left(|\Omega(\w)|\mu\right))^{1-\beta} + \left(|\Omega(\w)|\mu\right)^{\beta} \right) 
\end{align*}
and
\begin{align*}
wL(\w)w^{-1}\gamma^{1/2} & = L(\w)|\Omega(\w)|^{-1/2}\eps^{-1/2}\left(|\Omega(\w)|\mu + \left(|\Omega(\w)|\mu\right)^{2\beta} \right)^{1/2} \\
& \leq (\rho^*)^{1/2}\eps^{-1/2}\left(|\Omega(\w)|\mu + \left(|\Omega(\w)|\mu\right)^{2\beta} \right)^{1/2}. 
\end{align*}

Using these bounds in \eqref{eqn:neck4b}, together with $0<\beta<\tfrac{1}{2}$, implies that
\begin{align}\label{eqn:neck5}
    \#\{D \in \mathcal V_3^\delta(\eps, u): (\tilde A) \,\text{holds}\}\leq  C^*\eps^{-1}\left(\left(|\Omega(\w)|\mu\right)^{1-\beta} + \left(|\Omega(\w)|\mu\right)^{\beta}\right).
\end{align}
 Since $\beta < 1-\beta$ for $\beta<\tfrac{1}{2}$, the quantity in \eqref{eqn:neck5} therefore satisfies the estimate in the statement of the proposition.

\noindent \underline{\emph{Case 2: $D \in \mathcal{V}^\delta_3(\epsilon, u)$ is such that $(\tilde B)$ holds.}}

In this case, by  Lemma \ref{lem:neck-straighten2} there exist a neck $\mathscr{N}(\w):= \Q_{ij,k}(\w)$, of minimum width $w:=w_{ij,k}$, a Lipschitz set $V$  that is a subset of a strip of width $2w$ and length $C^*\delta$,
 and a function $v\in H^1(V)$ such that 
 \begin{align*}
    \Area(V) \leq C^*\Area(\{x\in D\cap \mathscr{N}(\w)\,:\, \chi_4^{\delta}u \neq0\})
\end{align*}
 and $\int_V v^2\geq c^*\epsilon$.
In particular,
we can find $x^*$ such that
\begin{align} \label{eqn:neck-straighten2b}
    \int_{V\cap\{x=x^*\}} v(x^*,y)^2\,dy \geq c^*(C^*)^{-1}\eps \delta^{-1}.
\end{align}
We now integrate to the left or right of the strip, as in the proof of Proposition \ref{prop:corner} from \eqref{eqn:corner2} to \eqref{eqn:corner-end1}, using \eqref{eqn:neck-straighten2a} and \eqref{eqn:neck-straighten2b} in place of \eqref{eqn:corner0b} and \eqref{eqn:corner2}, this time applying a cut-off function in the $x$-variable. This provides a Lipschitz set $W$ which is a subset of the infinite strip $\mathbb{R}\times[-w,w]$, and a function $\tilde{v}\in H^1(W)$ such that $\tilde{v}$ vanishes on the part of $\pa W$ in the interior of the strip. Moreover, 
$
\Area(W)\leq \Area(V),
$
and, for $(|\Omega(\w)|\mu)^{\beta}$ sufficiently large, the function $\tilde{v}$ satisfies
\begin{align*}
    \frac{\int_{W}|\nabla \tilde{v}|^2}{\int_{W}\tilde{v}^2} \leq C^*\eps^{-4}\delta^4\mu^3,
\end{align*}
for a constant $C^*$ that may increase from line-to-line. By reflecting across the line $y=w$, we therefore get a set $\tilde{W}$ on a cylinder of circumference $4w$, with  $\Area(\tilde W)=2\Area(W)$, and so that its first Dirichlet eigenvalue, $\lambda=\lambda_1(\tilde W)$, satisfies
\begin{align*}
    \lambda\leq C^*\eps^{-4}\delta^4\mu^3 =: C^*\gamma.
\end{align*}
 As 
 \begin{align*}
    \Area(\tilde W) \leq C^*\Area(\{x\in D\cap \mathscr{N}(\w)\,:\, \chi_4^{\delta}u \neq0\}),
\end{align*}
 by Lemma \ref{lem:cylinder} we obtain
$$
\Area(\{x\in D\cap \mathscr{N}(\w) :\chi_4^\delta u\neq0\})
\geq C^* \min\{\pi\lambda_1(\mathbb{D})\gamma^{-1},4w\lambda_1(\mathbb{D})^{1/2}\gamma^{-1/2}\}.
$$
By Definition \ref{defn:partition}, and the definition of the neck-width constant $w^*$, the area of the part of the neck $\Q(\w)$ in the support of $\chi_4^{\delta}$ is bounded by a constant $C^*$ multiplied by $w\delta$. Therefore,  
\begin{align} \label{eqn:neck-case2}
   \#\{D \in \mathcal V_3^\delta(\eps, u): (\tilde B) \,\text{holds}\} \leq    C^*w\delta\max\{\gamma,w^{-1}\gamma^{1/2}\}.
\end{align}
Setting $\delta = |\Omega(\w)|^{1/2-\beta}\mu^{-\beta}$ gives
\begin{align*}
    \gamma = \eps^{-4}|\Omega(\w)|^{2-4\beta}\mu^{3-4\beta} = \eps^{-4}|\Omega(\w)|^{-1}\left(\Omega(\w)\mu\right)^{3-4\beta}.
\end{align*}
Using $0<\beta<\tfrac{1}{2}$, and $w\leq 4\delta$, the quantity in \eqref{eqn:neck-case2}  therefore satisfies
\begin{align*}
    \#\{D \in \mathcal V_3^\delta(\eps, u): (\tilde B) \,\text{holds}\}  \leq    C^*\eps^{-4}((|\Omega(\w)|\mu)^{3-6\beta}+(|\Omega(\w)|\mu)^{3/2-3\beta}) .
\end{align*}
This satisfies the estimate given in the statement of the proposition, and finishes the proof.
\end{proof}

%%%%%%%%%%%%%%%%%%%%%%%%%%%%%%%%%%%%%%%%%
%%%%%%%%%%%%%%%%%%%%%%%%%%%%%%%%%%%%%%%%%
\section{Proof of Theorem \ref{thm:Main}} \label{sec:end}
%%%%%%%%%%%%%%%%%%%%%%%%%%%%%%%%%%%%%%%%%
%%%%%%%%%%%%%%%%%%%%%%%%%%%%%%%%%%%%%%%%%
This section is dedicated to the proof of Theorem \ref{thm:Main}. The upper bound on $\nu(u_m(\w))$, the number of nodal domains for the $m$-th eigenfunction $u_m(\w)$, will follow from the control on $\{\nu_j^\delta(\eps, u_m(\w))\}_{j=0}^3$ that we developed in previous sections for appropriately chosen $\delta, \eps$. At the same time, when $u_m(\w)$ is Courant sharp, we know that $\nu(u_m(\w))=m$ is bounded below by the number of Neumann eigenvalues under $\mu_m(\w)$. Therefore,  we prove Theorem \ref{thm:Main} in Section \ref{s:pf} after first establishing a lower bound on the Neumann counting function in Section \ref{sec:Weyl}.

%%%%%%%%%%%%%%%%%%%%%%%%%%%%%%%%%%%%%%%%%
\subsection{A lower bound on the Neumann counting function} \label{sec:Weyl}
%%%%%%%%%%%%%%%%%%%%%%%%%%%%%%%%%%%%%%%%%

In this section, we obtain a lower bound on the Neumann counting function
$
\mathcal{N}^N_{\Omega(\w)}(\mu) = \#\{j:\mu_j(\w)<\mu\},
$
using a Weyl remainder estimate.
\begin{prop} \label{prop:Weyl}
Let $\Omega(\w)$ be a chain domain. There exists ${C}^*>0$, such that
\begin{align*}
       \mathcal{N}^N_{\Omega(\w)}(\mu) -  \tfrac{1}{4\pi}|\Omega(\w)|\mu  \geq -C^* \left(|\Omega(\w)| \mu\right)^{3/4} 
\end{align*}
whenever $|\Omega(\w)|\mu\geq {C}^*$.
\end{prop}
\begin{proof1}{Proposition \ref{prop:Weyl}}
By the min-max characterization of eigenvalues $\mu_m(\w) \leq \lambda_m(\w)$, where $\lambda_m(\w)$ is the $m$-th Dirichlet eigenvalue of $\Omega(\w)$. Therefore,
\begin{align*}
    \mathcal{N}^N_{\Omega(\w)}(\mu)\geq \mathcal{N}^D_{\Omega(\w)}(\mu),
\end{align*}
where $\mathcal{N}^D_{\Omega(\w)}(\mu)$ is the Dirichlet counting function. Next, let
$$
R_{\Omega(\w)}(\mu) =   \tfrac{1}{4\pi}|{\Omega}(\w)|\mu -\mathcal{N}^D_{\Omega(\w)}(\mu).
$$
By \cite[Equation  (13)]{vdberg-gittins2016}, in dimension $2$, we obtain, for any $\eps>0$,
\begin{align} \label{eqn:Weyl1}
    R_{\Omega(\w)}(\mu) \leq \tfrac{1}{4\pi}M_{\Omega(\w)}(\sqrt{2}\eps)\mu + |\Omega(\w)| (\sqrt{2}\eps)^{-1}\mu^{1/2}.
\end{align}
For $t$ satisfying
\begin{align} \label{eqn:t-upper}
    t\leq \tfrac{3}{4}L(\w)\min\left\{\tau^*\delta^*,\tfrac{1}{\kappa^*}\right\},
\end{align} 
we can apply Lemma \ref{lem:M-bound}, so that 
\begin{align*}
    M_{\Omega(\w)}(t) \leq C^*L(\w)t,
\end{align*}
for a constant $C^*$. Therefore, setting $t = 2\sqrt{\eps}$, from \eqref{eqn:Weyl1} we obtain the estimate
\begin{align} \label{eqn:Weyl2a}
      R_{\Omega(\w)}(\mu) \leq C^*L(\w)t\mu  + |\Omega(\w)| t^{-1}\mu^{1/2},
\end{align}
provided \eqref{eqn:t-upper} continues to hold. Defining $t>0$ by
\begin{align} \label{eqn:t-upper2}
    t^2 = (C^*)^{-1}|\Omega(\w)|/L(\w) \mu^{-1/2},
\end{align}
in order to minimize the right hand side of \eqref{eqn:Weyl2a}, gives
\begin{align} \label{eqn:Weyl2}
     R_{\Omega(\w)}(\mu) \leq 2(C^*)^{1/2}L(\w)^{1/2}|\Omega(\w)|^{1/2}\mu^{3/4} \leq 2(C^*)^{1/2}(|\Omega(\w)|\mu)^{3/4}.
\end{align}
Note that from \eqref{eqn:t-upper2} and the isoperimetric inequality,
\begin{align*}
\Big(\frac{t}{L(\w)}\Big)^4 = \frac{(C^*)^{-2}(|\Omega(\w)|/L(\w)^2)^3}{|\Omega(\w)|\mu} \leq \frac{(C^*)^{-2}(4\pi)^{-3}}{|\Omega(\w)|\mu}.
\end{align*}
Therefore, there exists a constant $\tilde{C}^*$ such that, for $|\Omega(\w)|\mu \geq \tilde{C}^*$, this choice of $t$ satisfies \eqref{eqn:t-upper}. Since
\begin{align*}
    \mathcal{N}^N_{\Omega(\w)}(\mu) -  \tfrac{1}{4\pi}|\Omega(\w)|\mu \geq   \mathcal{N}^D_{\Omega(\w)}(\mu) -  \tfrac{1}{4\pi}|\Omega(\w)|\mu  = -R_{\Omega(\w)}(\mu),
\end{align*}
the estimate in \eqref{eqn:Weyl2} thus completes the proof of the proposition.
\end{proof1}

%%%%%%%%%%%%%%%%%%%%%%%%%%%%%%%%%%%%%%%%%
\subsection{Proof of Theorem \ref{thm:Main}}\label{s:pf}
%%%%%%%%%%%%%%%%%%%%%%%%%%%%%%%%%%%%%%%%%

We now combine the estimates on the nodal domain counts from previous section with  Proposition \ref{prop:Weyl} in order to prove Theorem \ref{thm:Main}.

Let $u_m(\w)$ be the $m$-th Neumann eigenfunction of $\Omega(\w)$ with eigenvalue $\mu_m(\w)$. To prove the theorem, we can assume that $|\Omega(\w)|\mu_m(\w)$ is sufficiently large so that the estimates in Propositions  \ref{prop:bulk}, \ref{prop:boundary}, \ref{prop:corner}, \ref{prop:neck}, and \ref{prop:Weyl} all hold with $\mu=\mu_m(\w)$. We will use these estimates to then show that  $|\Omega(\w)|\mu_m(\w)\leq C^*$ for a constant $C^*$. 

Let 
$$\beta := \tfrac{3}{8}, \qquad \delta:=|\Omega(\w)|^{1/2-\beta}\mu^{-\beta}.$$
By Propositions \ref{prop:bulk}, \ref{prop:boundary}, \ref{prop:corner}, and \ref{prop:neck}, there exists a constant $\tilde{C}^*$ such that for each $0<\eps<\tfrac{1}{2}$, by \eqref{e: nd bd} the number of nodal domains of $u_m(\w)$ satisfies
\begin{align} \label{eqn:end1}
    \nu(u_m(\w))
    &\leq\sum_{j=0}^3\nu_j^\delta(\eps, u_m(\w))\\
    &\leq \frac{1}{\pi\lambda_1(\mathbb{D})}\frac{1+\eps}{1-\eps}|\Omega(\w)|\mu_m(\w) + \tilde{C}^*\eps^{-4}\left(|\Omega(\w)|\mu_m(\w)\right)^{3/4},
\end{align}
provided that $|\Omega(\w)|\mu\geq \tilde{C}^*$. Since $\mathcal{N}^{N}_{\Omega(\w)}(\mu_m(\w)) \leq m-1$, by Proposition \ref{prop:Weyl},
\begin{align} \label{eqn:end2}
     m-1 -  \frac{1}{4\pi}|\Omega(\w)|\mu_m(\w)  \geq -C^*\left(|\Omega(\w)|\mu_m(\w)\right)^{3/4},
\end{align}
provided that $|\Omega(\w)|\mu\geq {C}^*$. 

When $u_m(\w)$ is Courant sharp, we have $\nu(u_m(\w)) = m$. Therefore, combining \eqref{eqn:end1} and \eqref{eqn:end2} implies, in the Courant sharp case, that $\mu_m(\w)$ satisfies
\begin{align} 
  \tfrac{1}{4\pi}|\Omega(\w)|\mu_m(\w) - C^*  \left(|\Omega(\w)|\mu_m(\w)\right)^{3/4} \leq \frac{1}{\pi\lambda_1(\mathbb{D})}\frac{1+\eps}{1-\eps}|\Omega(\w)|\mu_m(\w) + \tilde{C}^*\eps^{-4}\left(|\Omega(\w)|\mu_m(\w)\right)^{3/4}. \label{eqn:end3}
\end{align}
We fix a small absolute constant $\eps>0$ so that
\begin{align*}
    \frac{1}{\pi\lambda_1(\mathbb{D})}\frac{1+\eps}{1-\eps} < \frac{1}{4\pi},
\end{align*}
which we can do since $4/\lambda_1(\mathbb{D})<1$. Then,  we can rearrange \eqref{eqn:end3} to guarantee that $$|\Omega(\w)|\mu_m(\w) \leq C_1^*$$ for a constant depending only on $C^*$, $\tilde{C}^*$ and this value of $\eps$. This establishes the second part of the theorem. To prove the first part of the theorem, we see from \eqref{eqn:end1} that for any eigenfunction $u_m(\w)$, 
\begin{align*}
     \frac{\nu(u_m(\w))}{m} \leq \frac{1}{\pi\lambda_1(\mathbb{D})}\frac{1+\eps}{1-\eps}\,\frac{|\Omega(\w)|\mu_m(\w)}{m} + \tilde{C}^*\eps^{-4}\frac{\left(|\Omega(\w)|\mu_m(\w)\right)^{3/4}}{m} .%+ C\eps^{-4}m^{-1}.
\end{align*}
Using \eqref{eqn:end2}, we therefore have
\begin{align*}
    {\lim\sup}_{m\to\infty} \frac{\nu(u_m(\w))}{m} \leq \frac{4\pi}{\pi\lambda_1(\mathbb{D})}\frac{1+\eps}{1-\eps} 
\end{align*}
for any $\eps>0$. Letting $\eps\to0$ then completes the proof of the theorem.

\bibliographystyle{plain}
\bibliography{nodal}

\def\cprime{$'$}
\begin{thebibliography}{10}

\bibitem{arrieta95}
Jos\'e~M. Arrieta.
\newblock Rates of eigenvalues on a dumbbell domain. {S}imple eigenvalue case.
\newblock {\em Transactions of the American Mathematical Society},
  347(9):3503--3531, 1995.

\bibitem{beck2022nodal}
Thomas Beck, Marichi Gupta, and Jeremy~L Marzuola.
\newblock Nodal set openings on perturbed rectangular domains.
\newblock {\em arXiv preprint arXiv:2204.12009}, 2022.

\bibitem{berard-hellfer2016}
Pierre B\'{e}rard and Bernard Helffer.
\newblock The weak {P}leijel theorem with geometric control.
\newblock {\em J. Spectr. Theory}, 6(4):717--733, 2016.

\bibitem{berard-meyer1982}
Pierre B\'{e}rard and Daniel Meyer.
\newblock In\'{e}galit\'{e}s isop\'{e}rim\'{e}triques et applications.
\newblock {\em Ann. Sci. \'{E}cole Norm. Sup. (4)}, 15(3):513--541, 1982.

\bibitem{vdberg-gittins2016}
Michiel van~den Berg and Katie Gittins.
\newblock On the number of {C}ourant-sharp {D}irichlet eigenvalues.
\newblock {\em J. Spectr. Theory}, 6(4):735--745, 2016.

\bibitem{chavel84}
Isaac Chavel.
\newblock {\em Eigenvalues in Riemannian Geometry}.
\newblock Series in Pure and Applied Mathematics, Academic Press, 1984.

\bibitem{courant}
Richard Courant and David Hilbert.
\newblock {\em Methods of mathematical physics: partial differential
  equations}.
\newblock John Wiley \& Sons, 2008.

\bibitem{daners2003dirichlet}
Daniel Daners.
\newblock Dirichlet problems on varying domains.
\newblock {\em Journal of Differential Equations}, 188(2):591--624, 2003.

\bibitem{gittins-lena2020}
Katie Gittins and Corentin L\'{e}na.
\newblock Upper bounds for {C}ourant-sharp {N}eumann and {R}obin eigenvalues.
\newblock {\em Bull. Soc. Math. France}, 148(1):99--132, 2020.

\bibitem{grisvard2011elliptic}
Pierre Grisvard.
\newblock {\em Elliptic problems in nonsmooth domains}.
\newblock SIAM, 2011.

\bibitem{hassannezhad2023pleijel}
Asma Hassannezhad and David Sher.
\newblock On {P}leijel's nodal domain theorem for the {R}obin problem.
\newblock {\em arXiv preprint arXiv:2303.08094}, 2023.

\bibitem{helffer2009nodal}
Bernard Helffer, Thomas Hoffmann-Ostenhof, and Susanna Terracini.
\newblock Nodal domains and spectral minimal partitions.
\newblock {\em Ann. Inst. H. Poincar\'{e} Anal. Non Lin\'{e}aire},
  26(1):101--138, 2009.

\bibitem{howards1999}
Hugh Howards, Michael Hutchings, and Frank Morgan.
\newblock The isoperimetric problem on surfaces.
\newblock {\em The American Mathematical Monthly}, 106(5):430--439, 1999.

\bibitem{J89}
Shuichi Jimbo.
\newblock The singularly perturbed domain and the characterization for the
  eigenfunctions with {N}eumann boundary condition.
\newblock {\em J. Differential Equations}, 77(2):322--350, 1989.

\bibitem{Jim}
Shuichi Jimbo.
\newblock Perturbation formula of eigenvalues in a singularly perturbed domain.
\newblock {\em J. Math. Soc. Japan}, 45(2):339--356, 1993.

\bibitem{MJIm}
Shuichi Jimbo and Yoshihisa Morita.
\newblock Remarks on the behavior of certain eigenvalues on a singularly
  perturbed domain with several thin channels.
\newblock {\em Comm. Partial Differential Equations}, 17(3-4):523--552, 1992.

\bibitem{kac1966can}
Mark Kac.
\newblock Can one hear the shape of a drum?
\newblock {\em The American Mathematical Monthly}, 73(4P2):1--23, 1966.

\bibitem{lena2015courant}
Corentin L{\'e}na.
\newblock Courant-sharp eigenvalues of a two-dimensional torus.
\newblock {\em Comptes Rendus Mathematique}, 353(6):535--539, 2015.

\bibitem{lena-pleijel}
Corentin L\'{e}na.
\newblock Pleijel's nodal domain theorem for {N}eumann and {R}obin
  eigenfunctions.
\newblock {\em Ann. Inst. Fourier (Grenoble)}, 69(1):283--301, 2019.

\bibitem{peetre1957}
Jaak Peetre.
\newblock A generalization of {C}ourant's nodal domain theorem.
\newblock {\em Math. Scand.}, 5:15--20, 1957.

\bibitem{pleijel1956}
{\AA}ke Pleijel.
\newblock Remarks on {C}ourant's nodal line theorem.
\newblock {\em Comm. Pure Appl. Math.}, 9:543--550, 1956.

\bibitem{polterovich2009}
Iosif Polterovich.
\newblock Pleijel's nodal domain theorem for free membranes.
\newblock {\em Proc. Amer. Math. Soc.}, 137(3):1021--1024, 2009.

\bibitem{singer2017spectral}
Amit Singer and Hau-Tieng Wu.
\newblock Spectral convergence of the connection {L}aplacian from random
  samples.
\newblock {\em Information and Inference: A Journal of the IMA}, 6(1):58--123,
  2017.

\bibitem{trillos2018variational}
Nicolas~Garcia Trillos and Dejan Slep{\v{c}}ev.
\newblock A variational approach to the consistency of spectral clustering.
\newblock {\em Applied and Computational Harmonic Analysis}, 45(2):239--281,
  2018.

\bibitem{U73}
Karen Uhlenbeck.
\newblock The {M}orse index theorem in {H}ilbert space.
\newblock {\em J. Differential Geometry}, 8:555--564, 1973.

\bibitem{U76}
Karen Uhlenbeck.
\newblock Generic properties of eigenfunctions.
\newblock {\em Amer. J. Math.}, 98(4):1059--1078, 1976.

\bibitem{wigley1970}
Neil Wigley.
\newblock Mixed boundary value problems in plane domains with corners.
\newblock {\em Mathematische Zeitschrift}, 115:33--52, 1970.

\end{thebibliography}

\end{document}